\documentclass[reqno]{amsart}
\usepackage{amsmath,amsthm,amssymb}
\usepackage{color}

\makeatletter
    
    \@addtoreset{equation}{section}
  \makeatother

\newtheorem{Theorem}{Theorem}[section] 
\newtheorem{Definition}[Theorem]{Definition} 
\newtheorem{Corollary}[Theorem]{Corollary} 
\newtheorem{Lemma}[Theorem]{Lemma} 
\newtheorem{Proposition}[Theorem]{Proposition} 
\newtheorem{Remark}{Remark}[section]

\title[Estimates of lifespan for semilinear damped wave equations]{
Estimates of lifespan and blow-up rates for the wave equation
with a time-dependent damping and a power-type nonlinearity}

\author{Kazumasa Fujiwara}

\address{%
Department of Pure and Applied Physics,
Waseda University,
3-4-1, Okubo, Shinjuku-ku,
Tokyo, 169-8555, Japan}

\email{k-fujiwara@asagi.waseda.jp}

\author{Masahiro Ikeda}

\address{%
Department of Mathematics,
Graduate School of Science,
Kyoto University, Kyoto,
Kyoto 606-8502, Japan}
\email{mikeda@math.kyoto-u.ac.jp}

\author{Yuta Wakasugi}

\address{%
Graduate School of Mathematics,
Nagoya University, Furocho, Chikusaku,
Nagoya 464-8602, Japan
}
\email{yuta.wakasugi@math.nagoya-u.ac.jp}

\begin{document}
\maketitle

\begin{abstract}
We study estimates of lifespan and blow-up rates
of solutions for the Cauchy problem of
the wave equation with a time-dependent damping
and a power-type nonlinearity.
When the damping acts on the solutions effectively,
and the nonlinearity belongs to the subcritical case,
we show the sharp lifespan estimates and the blow-up rates of solutions.
The upper estimates are proved by an ODE argument, and
the lower estimates are given by a method of scaling variables.
\end{abstract}

\section{Introduction}
We consider the Cauchy problem of the
wave equation with a time-dependent damping
and a power-type nonlinearity
	\begin{align}
	\begin{cases}
	\Box u + b(t) u_t = |u|^{p},
	&\qquad t \in \lbrack 0,T), \quad x \in \mathbb R^n,\\
	u(0) = u_0,
	\quad u_t(0) = u_1,
	&\qquad x \in \mathbb R^n.
	\end{cases}
	\label{eq:1}
	\end{align}
Here
$u= u(t,x)$ is a real-valued unknown function,
$b(t)$
is a smooth positive function,
$\Box$ denotes $\partial_t^2 - \Delta_x$,
and
$u_0=u_0(x), u_1=u_1(x)$ are given initial data.

Damped wave equations are known as models
describing the voltage and the current on an electrical transmission line
with a resistance.
It is also derived as a modified heat conduction equation
from the heat balance law and the so-called
Cattaneo--Vernotte law
instead of the usual Fourier law (see \cite{Ca58, Li97, Ve58}).
The term $b(t)u_t$ is called the damping term,
which prevents the motion of the wave and reduces its energy,
and the coefficient $b(t)$ represents
the strength of the damping.

From a mathematical point of view,
it is an interesting problem to study
how the damping term affects the properties of the solution.
In particular, in this paper
we investigate
the relation between the damping term and
the blow-up behavior of the solution of the Cauchy problem \eqref{eq:1}.
To this end, as a typical case, we assume that
$b(t)$ satisfies
\begin{align}%
\label{b}
	b_1 (1+t)^{-\beta} \le b(t) \le b_2 (1+t)^{-\beta},\quad
	|b^{\prime}(t)| \le b_3 (1+t)^{-1}b(t)
\end{align}%
for $t \ge 0$
with some
$\beta \in \mathbb{R}$
and some positive constants
$b_1, b_2$ and $b_3$.

Since the the nonlinearity $|u|^p$ of \eqref{eq:1}
is a source term,
in general the solution may blow up in finite time
even if the initial data is sufficiently small.
Indeed, for the semilinear heat equation
$v_t - \Delta v = v^p$
with a nonnegative initial data
$v(0) = v_0 \ge 0$,
Fujita \cite{Fu66} found
that there is the critical exponent
$p_F = 1+2/n$,
that is,
if $p>p_F$, then the global solution uniquely exists
for suitably small initial data comparing with the Gaussian;
if $1<p< p_F$, then all positive solutions blow up in finite time.
Later on, it is shown that the critical case $p=p_F$
belongs to the blow-up region
(see Hayakawa \cite{Ha73} and Kobayashi, Sino and Tanaka \cite{KoSiTa77}).

The blow-up of solutions of semilinear damped wave equations
was firstly studied by
Li and Zhou \cite{LiZh95}.
They treated the so-called classical damping case,
that is, \eqref{eq:1} with $b(t) \equiv 1$,
and proved that when $n=1$ or $n=2$ and
$p\le p_F$,
if the initial data satisfy
$u_0,u_1\in C_0^{\infty}(\mathbb{R}^n)$
and
$\int_{\mathbb{R}^n} (u_0+u_1)(x)\,dx > 0$,
then the local solution blows up within a finite time.
Moreover, they obtained the sharp upper bound of the lifespan
in terms of the size of the initial data.
Namely, denoting
$u_0 = \varepsilon a_0, u_1 = \varepsilon a_1$
with
$\varepsilon > 0$ and
$a_0, a_1\in C_0^{\infty}(\mathbb{R}^n)$ having positive average,
they proved that the lifespan
(maximal existence time of the local solution)
$T_0$
satisfies
\begin{align}
\label{ls}
	T_0
	\le \begin{cases}
		C\varepsilon^{-\frac{1}{\frac{1}{p-1}-\frac{n}{2}}}&(1<p<p_F),\\
		\exp\left( C \varepsilon^{-(p-1)} \right)&(p=p_F).
		\end{cases}
\end{align}
Furthermore, by their method, we can prove
the estimate
\begin{align*}%
	I(t) \ge C \left( C I(0) -t \right)^{-\frac{2}{p-1}}
	\quad \mbox{with}\quad
	I(t) = \int_{\mathbb{R}^n} u(t,x) dx.
\end{align*}%
This shows the blow-up rate of the average of the solution.
However, their argument depends on the positivity of the
fundamental solution of the damped wave equation,
which is valid only in $n \le 3$,
and it cannot be applied to higher dimensional cases
or variable coefficient cases.
They also proved the global existence of solutions
for small initial data when $p>p_F$.
Therefore, they determined the critical exponent
for the classical damped wave equation for $n \le 2$
for smooth and compactly supported initial data.
Here, we say that
a number $p_c >1$ is the critical exponent for the
semilinear damped wave equation \eqref{eq:1}
if $p>p_c$, then the global solution uniquely exists for
sufficiently small data;
if $p\le p_c$, then the local solution blows up in finite time,
provided that the data has certain positive average determined from $b(t)$.

Later on,
for $n=3$,
Nishihara \cite{Ni03, Ni03Ib}
discovered a decomposition of the linear solution
\[
	S_n(t) u_1(x) = J_n(t) u_1(x) + e^{-\frac{t}{2}} W_n(t) u_1(x),
\]
where
$S_n(t), W_n(t)$ are the fundamental solution of the linear damped wave equation
$\Box u + u_t = 0$
and the linear wave equation
$\Box u = 0$, respectively,
and $J_n(t)u_1$ behaves as the solution of the linear heat equation
$v_t - \Delta v = 0$.
Then, he proved
the small data global existence when $p>p_F$
and
the sharp upper bound of the lifespan \eqref{ls}
when $p \le p_F$.
For $n=1,2$ and $n \ge 4$,
the same type decomposition was obtained by
Marcati and Nishihara \cite{MaNi03}, Hosono and Ogawa \cite{HoOg04}
and Narazaki \cite{Na04}
(see also Sakata and the third author \cite{SaWaMZ} for the exact decomposition for
$n \ge 4$).

For higher dimensional cases
$n \ge 4$,
Todorova and Yordanov \cite{ToYo01}
and Zhang \cite{Zh01}
determined the critical exponent as
$p = p_F$.

Concerning the estimate of the lifespan for $n \ge 4$,
for the subcritical case
$p<p_F$,
the second and the third author \cite{IkeWa15}
showed an almost sharp estimate of the lifespan
\begin{align*}
	C_1 \varepsilon^{-\frac{1}{\frac{1}{p-1}-\frac{n}{2}}+\delta}
	\le T_0
	\le C_2  \varepsilon^{-\frac{1}{\frac{1}{p-1}-\frac{n}{2}}}
\end{align*}
for small $\varepsilon > 0$
with arbitrary small $\delta>0$ and
some constants $C_1, C_2 >0$.
For the critical case
$p=p_F$,
the second author and Ogawa \cite{IkeOg16} obtained
\begin{align*}
	\exp\left( C_1 \varepsilon^{-(p-1)} \right)
	\le T_0
	\le \exp \left( C_2 \varepsilon^{-p} \right)
\end{align*}
with some constant $C_1, C_2 >0$
(see Proposition \ref{prop_IkOg} below).
As in the case $n \le 3$,
we expect that the sharp upper estimate of the lifespan
is given by
$T_0 \le \exp ( C \varepsilon^{-(p-1)} )$
for higher dimensional cases $n\ge 4$.
However, this problem is still open.

In regard to the lifespan estimate for the
semilinear wave equation with time-dependent damping
\eqref{eq:1},
much less is known.
When $b(t) = (1+t)^{-\beta}$
with $\beta \in (-1, 1)$
and $(u_0, u_1) \in H^1(\mathbb{R}^n) \times L^2(\mathbb{R}^n)$
is compactly supported,
Nishihara \cite{Ni11TJM}
and Lin, Nishihara and Zhai \cite{LiNiZh12}
proved that the critical exponent
is given by $p=p_F$
(see also D'Abbicco, Lucente and Reissig \cite{DaLuRe13}
for more general damping and initial data).
After that,
for subcritical cases
$p<p_F$,
the second author and the third author \cite{IkeWa15}
obtained an almost sharp estimate of the lifespan
\begin{align*}
	C_1 \varepsilon^{-\frac{1}{(\frac{1}{p-1}-\frac{n}{2})(1+\beta)}+\delta}
	\le T_0
	\le C_2  \varepsilon^{-\frac{1}{(\frac{1}{p-1}-\frac{n}{2})(1+\beta)}}
\end{align*}
with arbitrary small $\delta>0$ and
some constants $C_1, C_2 >0$.

For the case where $b(t) = b_0/(1+t)$ with $b_0 >0$, 
the linearized problem of \eqref{eq:1}
\begin{align*}
	\Box u + \frac{b_0}{1+t} u_t = 0
\end{align*}
has scaling invariance, and
it is known that the asymptotic behavior
of the solution depends on the value of the constant $b_0 > 0$
(see Wirth \cite{Wi04}).
For the semilinear problem
\begin{align}
\label{mu}
	\Box u + \frac{b_0}{1+t} u_t = |u|^p,
\end{align}
D'Abbicco and Lucente \cite{DaLu13} and
D'Abbicco \cite{Da15}
determined the critical exponent
as $p=p_F$ when
$b_0 \ge 5/3 \ (n=1)$, $b_0 \ge 3 \ (n=2)$ and $b_0 \ge n+2 \ (n\ge 3)$.
Moreover, in the special case $b_0 = 2$,
by setting $u = (1+t)w$,
the equation \eqref{mu} is transformed into
the semilinear wave equation
$\Box w = (1+t)^{-(p-1)}|w|^p$.
In view of this,
D'Abbicco, Lucente and Reissig \cite{DaLuRe15}
showed that the critical exponent is given by
$p_2(n) = \max\{ p_F, p_0(n+2) \}$ for $n\le 3$,
where
$p_0(m)$ is the positive root of
$(m-1)p^2 - (m+1)p -2 = 0$.
These results were recently extended to
scale-invariant mass and dissipation
by Nasciment, Palmieri and Reissig \cite{NasPaRe16}.
Wakasa \cite{Wak16} obtained the optimal
estimate of the lifespan of solutions to \eqref{mu} with $b_0 =2$ for $n=1$:
\begin{align*}
	T_0 \sim \begin{cases}
		C \varepsilon^{-\frac{p-1}{3-p}} & (p<3),\\
		\exp \left( C \varepsilon^{-(p-1)} \right) &(p=3).
		\end{cases}
\end{align*}
However, in the general case $b_0 \neq 2$,
the optimal lifespan estimate is not known,
while partial results were given in \cite{Wa_thesis}.

When $\beta = -1$,
the third author \cite{Wa17JMAA} recently studied the global existence
and asymptotic behavior for $p>p_F$.
However, there are no results about
blow-up and estimates of the lifespan
for $p \le p_F$.

Finally, for $\beta < -1$,
we expect that
for any $p>1$,
the global solution uniquely exists
for sufficiently small initial data.
This problem will be discussed elsewhere.
When $\beta >1$,
we expect that the critical exponent
is given by $p_0(n)$,
that is, the critical exponent coincides with that of
the semilinear wave equation without damping.
However, this is still an open problem.

In this paper, we give the sharp upper estimate of lifespan
for subcritical nonlinearities
$p<p_F$
and the effective damping
$\beta \in [-1, 1)$.
The case $\beta = -1$ is completely new.
We also prove the sharp lower estimate the lifespan
when $p=p_F$ and $\beta \in [-1,1)$.
For the case $\beta = 1$,
some upper estimates of the lifespan will be given,
while it seems not to be optimal in general.

To state our main results,
we first introduce the definition of strong solutions:

\begin{Definition}\label{def_sol}
Let
$(u_0, u_1) \in H^1(\mathbb{R}^n) \times L^2(\mathbb{R}^n)$
and let
$T \in (0,\infty]$.
A function
\[
	u \in C^2([0,T); H^{-1}(\mathbb R^n))
	\cap C^1([0,T); L^2 (\mathbb R^n)) \cap C([0,T); H^1 (\mathbb R^n))
\]
is called a strong solution of the Cauchy problem \eqref{eq:1}
on $[0, T)$
if $u$ satisfies the initial conditions
$u(0) = u_0, u_t(0) = u_1$
and the first equation of \eqref{eq:1} in $C^2([0,T) ; H^{-1}(\mathbb{R}^n))$.
\end{Definition}
When $1<p <\infty\ (n=1,2)$, $1<p \le n/(n-2)\ (n\ge 3)$,
for any $(u_0, u_1)\in H^1(\mathbb{R}^n) \times L^2(\mathbb{R}^n)$,
it is well-known that
there exist $T>0$ and
a unique strong solution $u$ of the Cauchy problem
\eqref{eq:1} on $[0,T)$
(see \cite{Str89} or \cite{NakOn93}).
We will show the existence of the strong solution
for some $T>0$ in the appendix for the reader's convenience.

\begin{Definition}\label{def_ls}
The lifespan $T_0$ of a solution $u$ for the Cauchy problem \eqref{eq:1}
is defined by
	\[
	T_0 = \sup \{T > 0 \mid u \mbox{ is a strong solution for \eqref{eq:1} on } [0,T) \}.
	\]
\end{Definition}

To state our main results,
we introduce assumptions and notations.
We recall $p_F = 1 + 2/n$.
In what follows, we assume that
the coefficient of the damping term $b(t)$ is a smooth function
satisfying \eqref{b} with some $\beta \in [-1,1]$.
We put
\[
	B(t) = \int_0^{t} b(\tau)^{-1}\,d\tau.
\]
Then, by the assumption \eqref{b}, $B(t)$ satisfies
\begin{align}%
\label{b_est}
	\begin{cases}
		B_1 (1+t)^{1+\beta} \le B(t) \le B_2 (1+t)^{1+\beta}&(\beta \in (-1,1]),\\
		B_1 \log (2+t) \le B(t) \le B_2 \log(2+t) &(\beta =-1)
	\end{cases}
\end{align}%
for $t \ge 1$ with some constants $B_1, B_2 >0$
(the second inequalities of each case are still valid for any $t \ge 0$).
Note that the function $B(t)$ is strictly increasing
due to the positivity of $b(t)$,
and hence, $B(t)$ has the inverse function $B^{-1}(\tau)$ satisfying
\begin{align}%
\label{b_inv}
	\begin{cases}
		B_3 (1+\tau)^{\frac{1}{1+\beta}}
			\le B^{-1}(\tau)
			\le B_4 (1+\tau)^{\frac{1}{1+\beta}} &( \beta \in (-1,1]),\\
		\exp \left( B_3 (1+\tau) \right)
			\le B^{-1}(\tau)
			\le \exp \left( B_4 (1+\tau) \right) &(\beta =-1)
	\end{cases}
\end{align}%
for $\tau \ge 1$ with some constants $B_3, B_4>0$
(the second inequalities of each case are still valid for any $\tau \ge 0$).
We also remark that the changing variable $s = B(t)$
transforms the corresponding parabolic problem
$b(t)v_t - \Delta v = |v|^p$
into
$\tilde{v}_s - \Delta \tilde{v} = |\tilde{v}|^p$
with $v(t,x) = \tilde{v}(B(t),x)$.
Therefore, the function $B(t)$ acts as a scaling function
for the time variable.

Next, we define $\widetilde \psi  \in C^\infty([0,\infty);[0,1])$ as
	\[
	\widetilde \psi (r) =
	\begin{cases}
	1 &\mbox{if}\quad r \leq 1,\\
	\searrow &\mbox{if}\quad 1 < r < 2,\\
	0 &\mbox{if}\quad r \geq 2.
	\end{cases}
	\]
Let $\psi : \mathbb R^n \ni x \mapsto \widetilde \psi (|x|)$
and let $\psi _R(x) = \psi (x/R)$.
For $p>1$, $\beta \in [-1,1]$ and $A>0$,
we define $\mu (p,b,\beta ,A)$ by
	\begin{align*}
	&\mu (p,b,\beta ,A) \\
	&= \min \left\{ 1, \frac{p-1}{2} b(0)A,
	\left[
	\frac{2(p+1)}{(p-1)^2}
	b_1^{-2} \left\{ 2^{\frac{1}{1+\beta}} (1+B_4) \right\}^{\max(0, 2 \beta )}
	+ \frac{2 (b_1^{-1}b_3+1)}{p-1} \right]^{-1}
	\right\}.
	\end{align*}
Here, we interpret the term
$\{ 2^{\frac{1}{1+\beta}} (1+B_4) \}^{\max(0, 2 \beta )} = 1$
if $\beta = -1$.
We also note that $A_1 \le A_2$ implies
$\mu(p,b,\beta,A_1) \le \mu(p,b,\beta,A_2)$.
The constant
$\mu(p,b,\beta,A)$
appears as the coefficient of the initial data
in the estimate of the lifespan
and the lower estimate of the average of the solution
(see Proposition \ref{Theorem:1.3}).
For $n \in \mathbb{N}$, $p>1$, $\ell \in \mathbb{N}$ satisfying $\ell > 2p'$,
and $\phi \in C_0^{\infty}(\mathbb{R}^n)$ with $\phi \ge 0$,
we also define $A(n,p,\ell,\phi )$ as
	\[
	A(n,p,\ell,\phi )
	= 2^{p'-1} p'^{-\frac{1}{p}} p^{\frac{1-p'}{p}}
	\| \Phi^{p'} \phi ^{\ell-2p'} \|_{L^{1}(\mathbb R^n)}^{\frac{1}{p}}
	\| \phi ^{\ell} \|_{L^1(\mathbb R^n)}^{\frac{1}{p'}},
	\]
where
$p' = p/(p-1)$ and
	\[
	\Phi
	= \phi ^{2-\ell} \Delta(\phi^\ell)
	= \ell(\ell-1) \nabla \phi \cdot \nabla \phi + \ell \phi  \Delta \phi .
	\]
We will derive an ordinary differential inequality
for the weighted average of the solution up to the constant $A(n,p,\ell,\phi)$.

Now we are in a position to give our main results.
The first one is the upper estimate of the lifespan of solutions to \eqref{eq:1}
in a general setting.
At the moment we do not need any condition on $p$ such as
$p \le p_F$ but we impose certain condition
with respect to the test function $\phi$.
\begin{Proposition}
\label{Theorem:1.3}
Let $\beta  \in [-1,1]$, $p \in (1, \infty)$
and $(u_0,u_1) \in H^1(\mathbb R^n) \times L^2(\mathbb R^n)$,
and let $u$ be the associated strong solution on
$[0,T_0)$ with the lifespan $T_0$.
Assume that there exists $\phi \in \mathcal S (\mathbb R^n;[0,\infty))$
such that
	\begin{align}
	0
	< I_\phi (0) - A(n,p,\ell,\phi )
	< 2^{\frac{1}{p-1}} \|\phi ^\ell\|_{L^1(\mathbb R^n)},
	\quad
	I_\phi '(0)
	> 0,
	\label{eq:2}
	\end{align}
where
	\[
	I_\phi (t)
	= \int_{\mathbb R^n} u (t,x) \phi ^\ell(x) dx.
	\]
Let
	\begin{align*}
	J_\phi (t)
	&= I_\phi (t) - A(n,p,\ell,\phi ),\\
	\widetilde J_\phi (0)
	&= 2^{-\frac{1}{p-1}} \|\phi ^\ell\|_{L^1(\mathbb R^n)}^{-1} J_\phi(0),\\
	A_1
	&= \frac{J_\phi '(0)}{J_\phi (0)}
	= \frac{I_\phi '(0)}{I_\phi (0) - A(n,p,\ell,\phi )}.
	\end{align*}
Then, we have
	\begin{align}
	J_\phi (t)
	\geq
	J_\phi (0)
	\bigg( 1 - \mu (p,b,\beta ,A_1)
	\widetilde J_\phi (0)^{p-1}
	B(t)
	\bigg)^{- \frac{2}{p-1}}
	\label{eq:rate}
	\end{align}
for $t \in [0,T_0)$.
Moreover,
the lifespan $T_0$ of the solution $u$ is estimated as
	\[
	T_0 \leq
	B^{-1} \left( \mu (p,b,\beta, A_1)^{-1} \widetilde J_\phi(0)^{1-p} \right).
	\]
\end{Proposition}

Proposition \ref{Theorem:1.3} implies that
\eqref{eq:2} is a sufficient condition for the blow-up of the solution.
The condition \eqref{eq:2} is related with the scaling of the equation.
Indeed, letting $p \in (1,p_F)$ and taking the test function
$\phi = \psi_{R(\varepsilon)}$
with an appropriate scaling parameter $R(\varepsilon)$,
we ensure the condition \eqref{eq:2} and
show the sharp estimate of the lifespan of solutions to \eqref{eq:1}.
\begin{Corollary}
\label{Theorem:1.4}
Let $\beta  \in [-1,1]$, $p \in (1,p_F)$
and $(u_0, u_1) = \varepsilon (a_0, a_1)$
with $\varepsilon > 0$.
We assume that
$(a_0, a_1)
\in (H^1(\mathbb R^n)\cap L^1(\mathbb{R}^n))
\times (L^2(\mathbb R^n) \cap L^1(\mathbb{R}^n))$
satisfy
	\[
		I_0 = \int_{|x|<R} a_0(x) dx > 0,
	\quad
		I_1 = \int_{|x|<R} a_1(x) dx > 0.
	\]
Let
	\begin{align}
	\label{r}
	R(\varepsilon)
	= A(n,p,\ell,\psi )^{\frac{p-1}{n(p_F-p)}}
	\bigg(\frac{\varepsilon}{4} I_0 \bigg)^{ - \frac{p-1}{n(p_F-p)}}
	\end{align}
and let $\varepsilon_0 >0$ satisfy that
\begin{align}
	&\int_{\mathbb R^n} \psi _{R(\varepsilon_0)}^{\ell}(x) a_0 (x)
	\ge \frac{1}{2} I_0,
	\label{eq:3}\\
	&\int_{\mathbb R^n} \psi _{R(\varepsilon_0)}^\ell(x) a_1(x) dx
	\geq \frac{1}{2} I_1,
	\label{eq:4}\\
	&\varepsilon_0 I_0
	\le 2^{\frac{1}{p-1}} \| \psi^{\ell} \|_{L^1(\mathbb{R}^n)} R(\varepsilon_0)^n.
	\label{eq:41}
\end{align}
Then, for any $\varepsilon \in (0,\varepsilon_0]$,
the associated strong solution $u$ of \eqref{eq:1} satisfies,
\begin{align}%
\label{bl-rt}
	\int_{\mathbb{R}^n} u(t,x) \psi_{R(\varepsilon)}^{\ell}(x) dx
	\ge \frac{\varepsilon}{4}I_0
		\left( 1- \mu_0
			\varepsilon^{\frac{1}{\frac{1}{p-1}-\frac{n}{2}}} B(t)
		\right)^{-\frac{2}{p-1}}
\end{align}%
with some constant $\mu_0 = \mu_0(n,p,b,\beta,\ell,\psi,I_0,I_1) > 0$
and the lifespan $T_0 = T_0(\varepsilon)$
is estimated as
	\begin{align}
	\label{lf_upp2}
	T_0 \le
	B^{-1} \left( \mu_0^{-1} \varepsilon^{-\frac{1}{\frac{1}{p-1}-\frac{n}{2}}} \right).
	\end{align}
\end{Corollary}

\begin{Remark}
The constant $\mu_0$ is given by \eqref{mu0}.
Also, under the assumptions in Corollary \ref{Theorem:1.4},
combining \eqref{lf_upp2} and \eqref{b_inv}, we see that
\begin{align*}%
	T_0 \le
	\begin{cases}
		B_4 \left(
				1+ \mu_0^{-1} \varepsilon^{-\frac{1}{\frac{1}{p-1}-\frac{n}{2}}}
			\right)^{\frac{1}{1+\beta}}
		&(\beta \in (-1,1]),\\
		\exp \left( B_4
			\left( 1+ \mu_0^{-1} \varepsilon^{-\frac{1}{\frac{1}{p-1}-\frac{n}{2}}} \right)
			\right)
		&(\beta = -1).
	\end{cases}
\end{align*}%
\end{Remark}

Propositions \ref{Theorem:1.3} and \ref{Theorem:1.4}
show the blow-up behavior for the solutions of \eqref{eq:1}
and how it depends on the parameter $\beta$.
They are summarized in the following way:

\begin{itemize}
\item
The blow-up rate of the solution near the blow-up time
is similar to
that of the nonlinear wave equation,
though the time variable is scaled by $B(t)$.
\item
On the other hand,
the estimate of the lifespan of the solution
is similar to
that of the nonlinear heat equation.
\end{itemize}

Concerning the upper estimate of the lifespan
$T_0$
in the critical case $p=p_F$,
we refer the reader to a recent result of
the second author and Ogawa \cite[Theorem 2.5]{IkeOg16}:
\begin{Proposition}[\cite{IkeOg16}]\label{prop_IkOg}
Let $b(t) = (1+t)^{-\beta}$, $\beta \in (-1,1)$, $p=p_F$ and
let $(u_0, u_1) = \varepsilon (a_0, a_1)$
with $\varepsilon > 0$, and we assume that
$(a_0, a_1) \in H^1(\mathbb{R}^n) \times L^2(\mathbb{R}^n)$
and $(a_0, a_1)$ satisfies
\[
	B_0 a_0 + a_1 \in L^1(\mathbb{R}^n)
	\quad \mbox{and} \quad
	\int_{\mathbb{R}^n} ( B_0 a_0 + a_1 )(x)\, dx > 0,
\]
where
\[
	B_0 =
	\left( \int_0^{\infty} \exp \left( -\int_0^t (1+s)^{-\beta}\,ds \right)\,dt \right)^{-1}.
\]
Then, there exists a constant
$C = C(n, \beta, a_0, a_1) >0$
such that the lifespan $T_0 = T_0(\varepsilon)$
of the associated strong solution is estimated as
\[
	T_0 \le \exp \left( C \varepsilon^{-p} \right)
\]
for any $\varepsilon \in (0,1]$.
\end{Proposition}

Next, we discuss the optimality of the estimate \eqref{lf_upp2}
with respect to the power of $\varepsilon$,
that is, the estimate of the lifespan from below.
Following the third author's recent work \cite{Wa17JMAA},
we have the lower estimate of the lifespan.

\begin{Proposition}\label{prop_low}
Let
$\beta \in [-1, 1)$, $p \in (1, p_F)$ and let
$(u_0, u_1) = \varepsilon (a_0, a_1)$
with $\varepsilon > 0$.
We assume that
$(a_0, a_1) \in H^{1,m}(\mathbb{R}^n) \times H^{0,m}(\mathbb{R}^n)$
with
$m = 1\ (n=1)$, $m > n/2 +1 \ (n\ge 2)$.
Then, there exist constants
$\varepsilon_1 = \varepsilon(n,\beta, p, m, \| a_0 \|_{H^{1,m}}, \| a_1 \|_{H^{0,m}} )> 0$
and
$C_{\ast} = C_{\ast} (n,\beta, p, m, \| a_0 \|_{H^{1,m}}, \| a_1 \|_{H^{0,m}} ) > 0$
such that
for any $\varepsilon \in (0,\varepsilon_1]$,
the lifespan $T_0 = T_0(\varepsilon)$
of the associated strong solution
is estimated by
\begin{align*}%
	T_0 \ge  B^{-1} \left( C_{\ast} \varepsilon^{-\frac{1}{\frac{1}{p-1}-\frac{n}{2}} } \right).
\end{align*}%
\end{Proposition}
\begin{Remark}
Under the assumptions in Proposition \ref{prop_low},
combining the above estimate and \eqref{b_inv}, we see that
\[
	T_0 \ge
	\begin{cases}
	\displaystyle C_{\ast} \varepsilon^{-\frac{1}{(\frac{1}{p-1}-\frac{n}{2})(1+\beta)} }
		& (\beta \in (-1,1)),\\
	\displaystyle \exp \left( C_{\ast} \varepsilon^{-\frac{1}{\frac{1}{p-1}-\frac{n}{2}} } \right)
		& (\beta = -1).
	\end{cases}
\]
\end{Remark}

In the case $\beta \in (-1,1)$,
the rate of $\varepsilon$ coincides with that of Corollary \ref{Theorem:1.4}.
Namely, we have the sharp estimate of the lifespan of the solution.
In the case $\beta = -1$, we have an exponential lower bound,
which is the so-called almost global existence of the solution.
This is quite reasonable because in this case where the damping
is very strong, it helps well the solution exist longer time.

On the other hand, in the critical case $p=p_F$, we have the following:
\begin{Proposition}\label{prop_lowcr}
Let $\beta \in [-1, 1)$, $p=p_F$
and let
$(u_0, u_1) = \varepsilon (a_0, a_1)$
with $\varepsilon >0$.
We assume that
$(a_0, a_1) \in H^{1,m}(\mathbb{R}^n) \times H^{0,m}(\mathbb{R}^n)$
with
$m = 1\ (n=1)$, $m>n/2+1\ (n \ge 2)$.
Then, there exist constants
$\varepsilon_2 = \varepsilon(n,\beta, p, m, \| a_0 \|_{H^{1,m}}, \| a_1 \|_{H^{0,m}} )> 0$
and
$C_{\ast} = C_{\ast} (n,\beta, p, m, \| a_0 \|_{H^{1,m}}, \| a_1 \|_{H^{0,m}} ) > 0$
such that
for any $\varepsilon \in (0,\varepsilon_2]$,
the lifespan $T_0 = T_0(\varepsilon)$
of the associated strong solution is estimated by
\begin{align*}%
	T_0 \ge B^{-1}\left( \exp \left( C_{\ast} \varepsilon^{-(p-1)} \right) \right).
\end{align*}%
\end{Proposition}
\begin{Remark}
Under the assumptions in Proposition \ref{prop_lowcr},
combining the above estimate and \eqref{b_inv}, we see that
\[
	T_0 \ge
	\begin{cases}
	\displaystyle
	\exp \left( C_{\ast}\varepsilon^{-(p-1)} \right)
		& ( \beta \in (-1,1)),\\
	\displaystyle
		\exp \left(
			\exp \left( C_{\ast} \varepsilon^{-(p-1)} \right)
			\right)
		& (\beta = -1).
	\end{cases}
\]
\end{Remark}

Proposition \ref{prop_lowcr} shows that
in the critical case, we have
the exponential and the double-exponential estimate
from the below for the case
$\beta \in (-1,1)$ and $\beta = -1$, respectively.
Comparing with Proposition \ref{prop_low},
it is also quite reasonable.
We also remark that Propositions \ref{prop_low} and Proposition \ref{prop_lowcr}
are still true if we replace the nonlinearity
$|u|^p$
by
$\pm |u|^{p-1}u$ or $-|u|^p$.

Our results with the previous ones
are summarized in Table 1,
where we consider the damping
$b_0 (1+t)^{-\beta}$
with $b_0 > 0$ and $\beta \in [-1,1]$.
\begin{table}[!h]
\begingroup
\renewcommand{\arraystretch}{1.8}
\begin{tabular}{|c|c|c|} \hline
	$\beta \backslash p$
		&  $\displaystyle 1< p < p_F$
		&  $\displaystyle p= p_F$ \\[4pt]
	\hline
	$\beta = -1$
		&$T_0 \sim
			\exp \left( C \varepsilon^{-\frac{1}{\frac{1}{p-1}-\frac{n}{2}} } \right)$
		&$\exp \left(
			\exp \left( C \varepsilon^{-(p-1)} \right)
			\right)
			\le T_0$\\
	\hline
	$-1<\beta <1$
	&$T_0 \sim C \varepsilon^{-\frac{1}{(\frac{1}{p-1}-\frac{n}{2})(1+\beta)} }$
	&$\exp \left( C \varepsilon^{-(p-1)} \right)
		\le T_0 \le \exp \left( C \varepsilon^{-p} \right)$ \ \\
	\ &\ &(upper bound is by \cite{IkeOg16})\\
	\hline
	$\beta = 1$
	&$T_0 \le C \varepsilon^{-\frac{1}{2(\frac{1}{p-1}-\frac{n}{2})} }$,
	&open (in general), \\
	\ 
	&$T_0 \sim \varepsilon^{-\frac{p-1}{3-p}}$
	&$T_0 \sim \exp \left( C \varepsilon^{-(p-1)} \right)$\\
	\ &for $n=1, b_0=2$ (\cite{Wak16})
	&for $n=1, b_0=2$ (\cite{Wak16})\\
	\hline
\end{tabular}
\caption{Estimates of lifespan}
\endgroup
\end{table}

When $1<p<p_F$ and $\beta =1$,
we have an upper bound of $T_0$,
while it seems not to be optimal in general.
In this case it is known that
the critical exponent may change
(see D'Abbicco, Lucente and Reissig \cite{DaLuRe15}
and Wakasa \cite{Wak16}).

We mention about the strategy of the proof.
Our method is a hybrid version of
the method by Li and Zhou \cite{LiZh95}
and by Zhang \cite{Zh01}.
The method by Li and Zhou \cite{LiZh95} is based on
a ordinary differential inequality.
However, in order to derive an ordinary differential inequality
from the damped wave equation,
their argument requires the positivity of the fundamental solution,
which fails in higher dimensional cases.
The method by Zhang \cite{Zh01} is the so-called
test function method.
He considered an average of the nonlinearity of the solution
$\int_0^{\infty}\int_{\mathbb{R}^n} |u(t,x)|^p \psi_R(t,x) dxdt$
with suitable family of cut-off functions
$\{ \psi_R \}_{R>0}$,
and leads to contradiction
by the integration by parts and a scaling argument.
However, this method is based on contradiction argument,
and the mechanism of the blow-up is unclear.
Moreover, by this approach, we cannot obtain
blow-up rates of solutions.

To overcome these difficulties,
we employ the method developed by the first author and Ozawa \cite{FuOz1, FuOz2}
in which the lifespan of the solution for
a nonlinear Schr\"{o}dinger equation is studied.
They considered a localized average of the solution
$\int_{\mathbb{R}^n} u(t,x) \phi(x) dx$
with a cut-off function $\phi (x)$,
and derive an ordinary differential inequality for it.
Then, they showed the estimate of the lifespan of
the solution of the ordinary differential inequality.
In this paper, we adapt their method to the damped wave equations.
First, we establish the blow-up and estimate of the lifespan
of the solution to the ordinary differential inequality
\begin{align}
\label{eq:6}
	\begin{cases}
	f''(t) + b(t) f'(t) \geq \gamma f(t)^p,\\
	f(0) \geq \varepsilon _0,\\
	f'(0) \geq A_0 \varepsilon _0,
	\end{cases}
\end{align}
with $A_0, \gamma, \varepsilon_0 >0$.
We remark that Li and Zhou also obtained the
finite time blow-up for \eqref{eq:6} and
the life-span estmate.
However, as far as the authors know,
explicit subsolutions for \eqref{eq:6}
had not been known even though
they are well known
for a first order ordinary differential inequality
$f' \geq f^p$.
We construct explicit subsolutions
by a comparison lemma given by Li and Zhou \cite{LiZh95}
and the blow-up rate \eqref{eq:rate} follows form these explicit subsolution.
For detail, see Proposition \ref{Theorem:2.3}.
Next, to prove Proposition \ref{Theorem:1.3},
we follow \cite{FuOz1, FuOz2} and consider the localized average
$I_{\phi}(t) = \int_{\mathbb{R}^n} u(t,x) \phi(x) dx$,
and derive the ordinary differential inequality \eqref{eq:6}
from the equation \eqref{eq:1}.
Finally, for Corollary \ref{Theorem:1.4},
we choose a special family of cut-off functions
and apply a scaling argument
to reduce its proof to Proposition \ref{Theorem:1.3}.

For Propositions \ref{prop_low} and \ref{prop_lowcr},
we employ the method of scaling variables,
which was originally introduced by
Gallay and Raugel \cite{GaRa98}.
Coulaud \cite{Co14} refined it and applied
to the second grade fluids equations in three space dimensions.
Recently, the third author \cite{Wa17JMAA} applied the method to
obtain the asymptotic profile for the
semilinear wave equation with time-dependent damping.

This paper is organized as follows.
In section 2,
we study the blow-up of solutions to
the ordinary differential inequality \eqref{eq:6}.
In section 3,
applying the theory of ordinary differential inequalities
prepared in Section 2, we give a proof of Proposition \ref{Theorem:1.3}
and Corollary \ref{Theorem:1.4}.
Section 4 is devoted to the proof of Propositions \ref{prop_low} and \ref{prop_lowcr}.
Finally, in the appendix, we give a proof of local existence
of solutions in the energy space.

We finish this section with notations used throughout this paper.
We denote by $C$ a generic constant, which may change
from line to line.
In particular, we sometimes use the symbol
$C(\ast,\ldots,\ast)$
for constants depending on the quantities appearing in parenthesis.

We give the notations of function spaces.
Let $L^p(\mathbb{R}^n)$ be the usual Lebesgue space
equipped with the norm
\begin{align*}%
	\| f \|_{L^p} = \left( \int_{\mathbb{R}^n} |f(x)|^p dx \right)^{1/p}
	\ (1\le p < \infty),\quad
	\| f \|_{L^{\infty}} = {\rm ess\,sup\,}_{x\in \mathbb{R}^n} |f(x)|.
\end{align*}%
For $s \in \mathbb{Z}_{\ge 0}, m \ge 0$,
we define the weighted Sobolev space $H^{s,m}(\mathbb{R}^n)$ by
\begin{align*}
	H^{s,m}(\mathbb{R}^n)
	&= \{ f \in L^2(\mathbb{R}^n) ; \| f \|_{H^{s,m}} < \infty \}, \\
	\| f \|_{H^{s,m}}
	&= \left( \sum_{|\alpha| \le s }
		\int_{\mathbb{R}^n} (1+|x|^2)^m |\partial_x^{\alpha} f (x) |^2\,dx \right)^{1/2}.
\end{align*}
In particular, when $m=0$, we also denote $H^{s,0}(\mathbb{R}^n)$ as $H^{s}(\mathbb{R}^n)$.
For an interval $I$
and a Banach space $X$,
we define $C^r(I; X)$ as the space of $r$-times
continuously differentiable mapping from
$I$ to $X$ with respect to the topology in $X$.


\section{Estimates of the lifespan of solutions to ordinary differential inequalities}
In this section,
we study the estimates of lifespan of solutions to
the ordinary differential inequality \eqref{eq:6}.
To this end, the following comparison theorem plays a critical role:

\begin{Lemma}{\cite[Lemma 3.1]{LiZh95}}
\label{Theorem:2.1}
Let $T>0$.
We assume that functions
$k, h \in C^2([0,T))$
satisfy
	\[
	\begin{cases}
	a(t) k''(t) + k'(t) \geq c(t) k(t)^p,\\
	a(t) h''(t) + h'(t) \leq c(t) h(t)^p
	\end{cases}
	\]
for $t \in [0,T)$,
where $p \geq 1$ and
$a(t), c(t)$
are nonnegative continuous function on $[0,T)$.
We further assume that
	\[
	\begin{cases}
	k(0) > h(0),\\
	k'(0) \geq h'(0).
	\end{cases}
	\]
Then, we have $k'(t) > h'(t)$ for any $t \in [0,T)$.
\end{Lemma}

Thanks to Lemma \ref{Theorem:2.1},
we analyze the behavior of solutions for \eqref{eq:6}
by comparing with subsolutions.
In the next lemma, we introduce our subsolution.

\begin{Lemma}
\label{Theorem:2.2}
Let $A_0 > 0$, $\beta \in [-1,1]$, $p>1$ and
let $\varepsilon _0 \in (0, 1]$.
We put
	\begin{align}
	\label{t1}
	T_1 =
	B^{-1} \left( \mu (p,b,\beta, A_0)^{-1} \varepsilon_0^{1-p} \right).
	\end{align}
Moreover, for $t \in \lbrack 0,T_1)$, we define
	\begin{align*}
	g(t)
	= \varepsilon_0
	\bigg( 1 - \mu (p,b,\beta ,A_0) \varepsilon _0^{p-1}
	B(t) \bigg)^{- \frac{2}{p-1}}.
	\end{align*}
Then $g$ satisfies that
	\[
	\begin{cases}
	g''(t) + b(t) g'(t) \leq g(t)^p,
	&\quad \mbox{for}\quad t \in \lbrack 0,T_1),\\
	g(0) = \varepsilon_0,\\
	g'(0) \leq A_0 \varepsilon_0.
	\end{cases}
	\]
\end{Lemma}

\begin{proof}
For simplicity, we denote $\mu(p,b,\beta,A_0)$ as $\mu$.
Since $\mu \leq \frac{p-1}{2} b(0) A_0$,
by a direct calculation, we have
	\begin{align*}
	g'(t)
	&= \frac{2 \mu }{p-1} \varepsilon _0^p
	\bigg( 1 - \mu \varepsilon _0^{p-1} B(t) \bigg)^{- \frac{p+1}{p-1}}
	b(t)^{-1},\\
	g'(0)
	&= \frac{2 \mu }{p-1} b(0)^{-1}\varepsilon_0^p
	\leq A_0 \varepsilon_0,\\
	g''(t)
	&= -\frac{2 \mu }{p-1} \varepsilon _0^p
	\bigg( 1 - \mu \varepsilon _0^{p-1}
	B(t) \bigg)^{- \frac{p+1}{p-1}}
	b'(t) b(t)^{-2}\\
	&\quad + \frac{2 (p+1)}{(p-1)^2} \mu^2 \varepsilon_0^{2p-1}
	\bigg( 1 - \mu \varepsilon _0^{p-1} B(t)
	\bigg)^{- \frac{2p}{p-1}} b(t)^{-2}.
	\end{align*}
Then, for $t < T_1$, we obtain
	\begin{align*}
	&g''(t) + b(t) g'(t)\\
	&\leq g(t)^p
	\bigg(
	\frac{2(p+1)}{(p-1)^2} \varepsilon_0^{p-1} b(t)^{-2} \mu^2
	+ \frac{2 }{p-1} b'(t)b(t)^{-2}\mu + \frac{2}{p-1} \mu \bigg)\\
	&\leq g(t)^p
	\bigg(
	\frac{2(p+1)}{(p-1)^2}
	2^{\frac{2\beta}{1+\beta}} b_1^{-2} (1+B_4)^{\max(0,2 \beta )}
	+ \frac{2 ( b_1^{-1}b_3 +1)}{p-1} \bigg) \mu\\
	&\leq g(t)^p.
	\end{align*}
Here, for the second inequality, when
$\beta \in (0,1]$,
we have used that
	\begin{align*}
	\varepsilon_0^{p-1} b(t)^{-2} \mu
	&\le \varepsilon_0^{p-1}\mu  b_1^{-2} (1 + T_1)^{2\beta} \\
	&\le \varepsilon _0^{p-1} \mu b_1^{-2}
		\left[
			1+ B^{-1} \left( \mu^{-1} \varepsilon_0^{1-p} \right) \right]^{2\beta} \\
	&\le \varepsilon _0^{p-1} \mu b_1^{-2}
		\left[ 1 + 
			B_4 \left( 1+\mu^{-1} \varepsilon_0^{1-p} \right)^{\frac{1}{1+\beta}}
		\right]^{2\beta} \\
	&\le \varepsilon _0^{p-1} \mu b_1^{-2} (1+B_4)^{2\beta}
		\left( 1+ \mu^{-1} \varepsilon_0^{1-p} \right)^{\frac{2\beta}{1+\beta}} \\
	&\le 2^{\frac{2\beta}{1+\beta}}
		\left( \varepsilon _0^{p-1} \mu \right)^{\frac{1-\beta}{1+\beta}} b_1^{-2}
			(1+B_4)^{2\beta} \\
	&\leq 2^{\frac{2\beta}{1+\beta}} b_1^{-2}(1+B_4)^{2\beta},
	\end{align*}
	and for the third inequality we have used the definition of
	$\mu(p,b,\beta,A_0)$.
\end{proof}

\begin{Proposition}
\label{Theorem:2.3}
Let $T_0>0$,
$A_0 > 0$,
$\beta \in [-1,1]$,
$p>1$,
$\gamma  > 0$,
and let
$\varepsilon _0 \in (0,\gamma^{-\frac{1}{p-1}}]$.
Assume that $f \in C^2([0,T_0))$ satisfies
$f(t) > 0$ for $t\in [0,T_0)$ and
	\[
		\begin{cases}
		f''(t) + b(t) f'(t) \geq \gamma  f(t)^p \quad \mbox{for}\quad t \in [0,T_0),\\
		f(0) > \varepsilon _0,\\
		f'(0) \geq A_0 \varepsilon _0.
		\end{cases}
	\]
Then,
with $\delta_0 = \gamma^{\frac{1}{p-1}} \varepsilon _0$,
we have
	\begin{align*}
	f (t)
	\geq
	\varepsilon _0 \bigg( 1
	- \mu (p,b,\beta,A_0) \delta_0^{p-1}
	B(t) \bigg)^{- \frac{2}{p-1}}
	\end{align*}
for $t \in [0,T_0)$,
and $T_0$ is estimated as
	\begin{align}
	\label{t0}
		T_0 \le 
		B^{-1} \left( \mu (p,b,\beta, A_0)^{-1} \delta_0^{1-p} \right) .
	\end{align}
\end{Proposition}

\begin{proof}
Let $\widetilde f = \gamma^{\frac{1}{p-1}} f$
and $\delta_0 = \gamma^{\frac{1}{p-1}} \varepsilon _0$.
Then, $\widetilde f$ satisfies
	\[
		\begin{cases}
		\widetilde f''(t) + b(t) \widetilde f'(t) \geq \widetilde f(t)^p
		\quad \mbox{for}\quad t\in [0,T_0),\\
		\widetilde f(0) > \delta_0,\\
		\widetilde f'(0) \geq A_0 \delta_0.
		\end{cases}
	\]
Let $T_1$ be defined in \eqref{t1}
with $\varepsilon_0 = \delta_0$,
that is,
$T_1$
is the right-hand side of \eqref{t0}.
For $\rho \in [0,1)$,
we put $\delta_\rho  = (1-\rho ) \delta_0$
and define
	\[
	\widetilde g_\rho (t)
	= \delta_\rho 
	\bigg( 1 - \mu (p,b,\beta,A_0) \delta_\rho ^{p-1} B(t)
	\bigg)^{- \frac{2}{p-1}}
	\]
for $t \in \lbrack 0,T_1)$.
Noting
$\delta_0 \in (0, \mu(p,b,\beta,A_0)^{-\frac{1}{p-1}}]$
and applying Lemma \ref{Theorem:2.2},
we see that
$\widetilde g_\rho$
satisfies
	\[
	\begin{cases}
	\widetilde g_\rho ''(t) + b(t) \widetilde g'_\rho (t)
	\leq \widetilde g_\rho (t)^p
	&\ \mbox{for}\quad  t \in \lbrack 0,T_1),\\
	\widetilde g_\rho (0) = \delta_\rho,\\
	\widetilde g_\rho '(0) \leq A_0 \delta_\rho.
	\end{cases}
	\]
We put $T_2 = \min (T_0, T_1)$.
Then, by Lemma \ref{Theorem:2.1},
for any $\rho  \in (0,1)$,
we have $\widetilde f(t) \geq \widetilde g_\rho (t)$ for
$t \in \lbrack 0, T_2 )$.
Noting the continuity of
$\widetilde g_{\rho}$
with respect to $\rho \in [0,1)$
and letting $\rho \to 0$, we see that
$\widetilde f(t) \geq \widetilde g_0(t)$ holds for any $t \in \lbrack 0,T_2)$.

Next, we see that
$T_2 = T_0$.
Indeed, if $T_0 > T_1$, namely $T_2 = T_1$, then
$\tilde{f}(t)$ is defined as a $C^2$ function on the interval $[0,T_1]$.
However, by the definition of
$\widetilde g_{0}$,
we immediately obtain
$\lim_{t \to T_1-0} \widetilde g_{0}(t) = \infty$.
This and the fact
$\widetilde f(t) \geq \widetilde g_0(t)$ for $t \in \lbrack 0,T_1)$
imply $\lim_{t\to T_1-0} \tilde{f}(t) = \infty$,
which contradicts $\tilde{f} \in C^2 ([0,T_1])$.
Consequently, we have $T_2 = T_0$, namely $T_0 \le T_1$
and we complete the proof.
\end{proof}


\section{Proof of the Proposition \ref{Theorem:1.3} and Corollary \ref{Theorem:1.4}}
\begin{proof}[Proof of Proposition \ref{Theorem:1.3}]
Let $u$ be a strong solution of \eqref{eq:1} on $[0,T_0)$
with the lifespan $T_0$.
Let $\phi \in \mathcal{S}(\mathbb{R}^n; [0,\infty))$
satisfy the inequality \eqref{eq:2}.
Recall that
$I_\phi (t) = \int_{\mathbb R^n} u(t,x) \phi^\ell (x) dx$
and
$\Phi(x) = \ell(\ell-1) \nabla \phi  (x) \cdot \nabla \phi  (x) + \ell \phi  (x) \Delta \phi  (x)$.
Then,
by the continuity of $I_{\phi}(t)$ with respect to $t$,
there exists $t_0> 0$ such that
$I_\phi (t) - A(n,p,\ell,\phi ) > 0$ holds for $t \in [0,t_0)$.
By a direct calculation, we have for $t \in [0,t_0)$,
	\begin{align*}
	\frac{d^2}{dt^2} I_\phi  (t) + b(t) \frac{d}{dt} I_\phi  (t)
	&= \int_{\mathbb R^n}
	(\partial_t^2 + b(t) \partial_t) u(t,x) \phi ^\ell(x) dx\\
	&= \int_{\mathbb R^n} u(t,x) \Delta (\phi ^\ell(x)) dx
	+ \| u(t) \phi ^{\frac{\ell}{p}} \|_{L^p(\mathbb R^n)}^p\\
	&= \int_{\mathbb R^n} u(t,x) \Phi(x) \phi ^{\ell-2}(x) dx
	+ \| u(t) \phi ^{\frac{\ell}{p}} \|_{L^p(\mathbb R^n)}^p\\
	&\geq - \| \Phi \phi ^{\frac{\ell}{p'}-2} \|_{L^{p'}(\mathbb R^n)}
	\| u(t) \phi ^{\frac{\ell}{p}} \|_{L^p(\mathbb R^n)}
	+ \| u(t) \phi ^{\frac{\ell}{p}} \|_{L^p(\mathbb R^n)}^p\\
	&\geq - 2^{\frac{p'}{p}} p'^{-1} p^{-\frac{p'}{p}}
	\| \Phi \phi ^{\frac{\ell}{p'}-2} \|_{L^{p'}(\mathbb R^n)}^{p'}
	+ 2^{-1} \| u(t) \phi ^{\frac{\ell}{p}} \|_{L^p(\mathbb R^n)}^p\\
	&\geq - 2^{p'-1} p'^{-1} p^{1-p'}
	\| \Phi^{p'} \phi ^{\ell-2p'} \|_{L^{1}(\mathbb R^n)}
	+ 2^{-1} \| \phi ^{\ell} \|_{L^1(\mathbb R^n)}^{1-p} I_\phi (t)^p\\
	&= 2^{-1} \| \phi ^{\ell} \|_{L^1(\mathbb R^n)}^{1-p}
	\big( I_\phi (t)^p - A(n,p,\ell,\phi )^p \big)\\
	&\geq 2^{-1} \| \phi ^{\ell} \|_{L^1(\mathbb R^n)}^{1-p}
	\big( I_\phi (t) - A(n,p,\ell,\phi ) \big)^p.
	\end{align*}
Here we note that the above inequality holds
as long as
$I_\phi (t) - A(n,p,\ell,\phi ) > 0$.
The above inequality implies that
$J_\phi (t) = I_\phi (t) - A(n,p,\ell,\phi )$
satisfies
	\begin{align}
	\label{Jphi_ode}
	\begin{cases}
	J_\phi ''(t) + b(t) J'_\phi (t)
		\geq 2^{-1} \| \phi ^{\ell} \|_{L^1(\mathbb R^n)}^{1-p} J_{\phi}(t)^p
	&\mbox{for}\quad t \in [0,t_0),\\
	J_\phi (0) = I_\phi (0) - A(n,p,\ell,\phi ),\\
	J_\phi '(0) = A_1 J_\phi (0),\\
	\end{cases}
	\end{align}
where
$A_1 = I'_{\phi}(0)/( I_{\phi}(0) - A(n,p,\ell,\phi))$,
which is a positive constant thanks to the assumption \eqref{eq:2}.
Moreover, by the assumption \eqref{eq:2},
we have
$I_\phi (0) - A(n,p,\ell,\phi ) \leq 2^{\frac{1}{p-1}} \| \phi ^{\ell} \|_{L^1(\mathbb R^n)}$.
Thus, we apply Proposition \ref{Theorem:2.3}
with
$\varepsilon_0 = I_\phi (0) - A(n,p,\ell,\phi )$,
$\gamma = 2^{-1} \| \phi ^{\ell} \|_{L^1(\mathbb R^n)}^{1-p}$
and
$f(t) = J_{\phi}(t)$
to obtain
\begin{align}%
\label{Jphi_est}
	J_{\phi}(t)
	\ge J_{\phi}(0)
		\left( 1
		- \mu (p,b,\beta,A_1)
		\tilde{J}_{\phi}(0)^{p-1}
		B(t) \right)^{- \frac{2}{p-1}}
\end{align}%
for $t \in [0,t_0)$,
where
$\tilde{J}_{\phi}(0) = 2^{-\frac{1}{p-1}} \| \phi^{\ell} \|_{L^1(\mathbb{R}^n)}^{-1} J_{\phi}(0)$.

Next, we show that
$J_{\phi}(t) > 0$
holds for any $t\in [0,T_0)$.
Indeed, if $J_{\phi}(t_{\ast}) = 0$ holds for some $t_{\ast} \in (0, T_0)$
and $J_{\phi}(t) > 0$ holds for $t \in [0,t_{\ast})$,
then, applying the same argument above,
we can prove the estimate \eqref{Jphi_est} for $t \in [0, t_{\ast})$.
However, the right-hand side of \eqref{Jphi_est} remains positive
for $t = t_{\ast}$, which contradicts $J_{\phi}(t_{\ast}) = 0$.
Thus, we have $J_{\phi}(t) > 0$ for any $t\in [0,T_0)$,
and $J_{\phi}(t)$ also satisfies the estimate \eqref{Jphi_est} for $t\in [0,T_0)$.
Hence, Proposition \ref{Theorem:2.3}
with $\delta_0 = \widetilde{J}_{\phi}(0)$ gives the desired estimates for
$J_{\phi}(t)$ and $T_0$.
\end{proof}

\begin{proof}[Proof of Corollary \ref{Theorem:1.4}]
Let $R(\varepsilon)$ be given by \eqref{r}.
Since $n - 2 \frac{p'}{p} = \frac{n(p - p_F)}{p-1}$,
we calculate
	\[
	A(n,p,\ell,\psi _{R(\varepsilon)})
	= A(n,p,\ell,\psi ) R(\varepsilon)^{\frac{n(p - p_F)}{p-1}}
	= \frac{\varepsilon}{4} I_0.
	\]
From this and the assumption \eqref{eq:3}, we obtain
\begin{align*}%
	I_{\psi_{R(\varepsilon)}} - A(n,p,\ell,\psi _{R(\varepsilon)})
	\ge \frac{\varepsilon}{4} I_0.
\end{align*}%
Also, the assumption \eqref{eq:4} immediately implies
$I'_{\psi_{R(\varepsilon)}} \ge \frac{\varepsilon}{2}I_1$.
Finally, the assumption \eqref{eq:41} leads to
\begin{align*}%
	I_{\psi_{R(\varepsilon)}} - A(n,p,\ell,\psi _{R(\varepsilon)})
	&\le \varepsilon I_0 \\
	&\le 2^{\frac{1}{p-1}} \| \psi^{\ell} \|_{L^1(\mathbb{R}^n)} R(\varepsilon)^n \\
	&= 2^{\frac{1}{p-1}} \|\psi_{R(\varepsilon)}^{\ell} \|_{L^1(\mathbb R^n)}.
\end{align*}%
Therefore, the condition \eqref{eq:2} is fulfilled and
Proposition \ref{Theorem:1.3} with $\phi = \psi_{R(\varepsilon)}$ implies that
\begin{align*}%
	J_{\psi_{R(\varepsilon)}}(t)
	\ge J_{\psi_{R(\varepsilon)}}(0)
		\left( 1 - \mu(p,b,\beta, A_1(\varepsilon)) \widetilde{J}_{\psi_{R(\varepsilon)}}(0)^{p-1}
			B(t) \right)
\end{align*}%
and the lifespan $T_0$ is estimated as
\begin{align*}%
	T_0 \le B^{-1}\left( \mu(p,b,\beta,A_1(\varepsilon))^{-1}
		\widetilde{J}_{\psi_{R(\varepsilon)}}(0)^{1-p} \right),
\end{align*}%
where
\begin{align*}%
	A_1(\varepsilon)
	= \frac{I'_{\psi_{R(\varepsilon)}}(0)}{I'_{\psi_{R(\varepsilon)}}(0)-A(n,p,\ell,\psi_{R(\varepsilon)})}.
\end{align*}%
Now, we again use the assumptions \eqref{eq:3} and \eqref{eq:4} to obtain
\begin{align*}%
	A_1(\varepsilon)
		\ge \frac{\frac{\varepsilon}{2}I_1}{\varepsilon I_0 - \frac{\varepsilon}{4}I_0}
		= \frac{2I_1}{3I_0}.
\end{align*}%
Moreover, we calculate
\begin{align*}%
	\widetilde{J}_{\psi_{R(\varepsilon)}}(0)^{p-1}
	&= 2^{-1} \| \psi_{R(\varepsilon)}^{\ell} \|_{L^1(\mathbb{R}^n)}^{-(p-1)}
		J_{\psi_{R(\varepsilon)}}(0)^{p-1} \\
	&\ge 2^{-1} \| \psi^{\ell} \|_{L^1(\mathbb{R}^n)}^{-(p-1)} R(\varepsilon)^{-n(p-1)}
		\left( \frac{\varepsilon}{4} I_0 \right)^{p-1}\\
	&= 2^{-1} \| \psi^{\ell} \|_{L^1(\mathbb{R}^n)}^{-(p-1)}
		A(n,p,\ell, \psi)^{-\frac{(p-1)^2}{p_F-p}}
		\left( \frac{\varepsilon}{4} I_0 \right)^{\frac{1}{\frac{1}{p-1}-\frac{n}{2}}}.
\end{align*}%
Consequently, letting
\begin{align}%
\nonumber
	&\mu_0(n,p,b,\beta,\ell,\psi,I_0,I_1)\\
\label{mu0}
	&= \mu \left(p,b,\beta,\frac{2I_1}{3I_0} \right)2^{-1}
			\| \psi^{\ell} \|_{L^1(\mathbb{R}^n}^{-(p-1)}
			A(n,p,\ell,\psi)^{-\frac{(p-1)^2}{p_F-p}}
			\left( \frac{1}{4}I_0 \right)^{\frac{1}{\frac{1}{p-1}-\frac{n}{2}}},
\end{align}%
we have the desired estimates \eqref{bl-rt} and \eqref{lf_upp2}.
\end{proof}

\section{Proofs of Propositions \ref{prop_low} and \ref{prop_lowcr}}
\subsection{Scaling variables, local existence and spectral decompostion}\ 

We give proofs of Propositions \ref{prop_low} and \ref{prop_lowcr}.
Sections 4.1--4.3 are almost the same as in \cite{Wa17JMAA}
and we present only their outlines.
Following \cite{Wa17JMAA}, we introduce the scaling variables
\begin{align}
\label{sc}
	s = \log (B(t) + 1 ),\quad y = (B(t) + 1 )^{-1/2}x.
\end{align}
Also, we use the notation
$t(s) = B^{-1}(e^s-1)$.
We change the coordinate and the unknown function as
\begin{align}
\label{uvw}
	\begin{array}{l}
	\displaystyle u(t,x) = (B(t)+1)^{-n/2}v(\log(B(t)+1), (B(t)+1)^{-1/2}x),\\[5pt]
	\displaystyle u_t(t,x) = b(t)^{-1}(B(t)+1)^{-n/2-1}w(\log(B(t)+1), (B(t)+1)^{-1/2}x).
	\end{array}
\end{align}
Then, the equation \eqref{eq:1} is transformed into the first order system
\begin{align}
\label{eq_vw}
	\left\{\begin{array}{ll}
	\displaystyle v_s-\frac{y}{2}\cdot \nabla_yv - \frac{n}{2}v = w,&s>0, y\in \mathbb{R}^n,\\[7pt]
	\displaystyle
	\frac{e^{-s}}{b(t(s))^2}\left( w_s-\frac{y}{2}\cdot \nabla_yw -\left(\frac{n}{2}+1\right)w \right)+w
		 = \Delta_yv+r(s,y),&s>0,y\in \mathbb{R}^n,\\[7pt]
	\displaystyle v(0,y) = v_0(y) = \varepsilon a_0(y),\ 
				w(0,y) = w_0(y) = \varepsilon a_1(y),
				&y\in \mathbb{R}^n,
	\end{array}\right.
\end{align}
where
\begin{align}
\label{r}
	r(s,y) &= \frac{b^{\prime}(t(s))}{b(t(s))^2}w
		+ e^{\frac{n}{2}(p_F-p) s} |v|^p.
\end{align}

The local well-posedness for the system \eqref{eq_vw}
was obtained by \cite[Proposition 3.6]{Wa17JMAA}.
In this paper, the solution satisfying certain integral equation
is constructed (mild solution).
Such solution also satisfies the condition of
our strong solution (see Definition \ref{def_sol}).
\begin{Proposition}{\cite[Proposition 3.6]{Wa17JMAA}}\label{prop_loc2}
There exists $S>0$ depending only on
$\| (v_0, w_0) \|_{H^{1,m}\times H^{0,m}}$
(the size of the initial data) such that
the Cauchy problem \eqref{eq_vw} admits a unique strong solution
$(v,w)$ satisfying
\[
	(v,w) \in C([0,S);H^{1,m}(\mathbb{R}^n)\times H^{0,m}(\mathbb{R}^n)).
\]
Also, if
$(u_0, u_1) \in H^{2,m}(\mathbb{R}^n) \times H^{1,m}(\mathbb{R}^n)$,
then the solution $(v,w) $ satisfies
\begin{align}
\label{vwcls2}
	(v,w) \in C([0,S);H^{2,m}(\mathbb{R}^n)\times H^{1,m}(\mathbb{R}^n))
		\cap C^1([0,S); H^{1,m}(\mathbb{R}^n) \times H^{0,m}(\mathbb{R}^n)).
\end{align}
Moreover, for arbitrarily fixed time $S^{\prime}>0$,
we can extend the solution to the interval $[0,S^{\prime}]$
by taking $\varepsilon$ sufficiently small.
Furthermore, if the lifespan
\[
	S_0 = S_0(\varepsilon) = \sup\{
		S\in (0,\infty) ;
		\mbox{there exists a unique strong solution}\ (v,w)\ \mbox{to \eqref{eq_vw}} \}
\]
is finite, then $(v,w)$ satisfies
$\lim_{s \to S_0} \| (v,w)(s) \|_{H^{1,m}\times H^{0,m}} = \infty$.
\end{Proposition}

Next, to obtain an a priori estimate for
$(v,w)$, we decompose
$(v,w)$ into the leading terms and the remainder terms.
Let $\alpha(s)$ be
\begin{align}
\label{alpha}
	\alpha(s) = \int_{\mathbb{R}^n}v(s,y)dy,
\end{align}
which is well-defined due to $v(s) \in H^{1, m}$ with $m>n/2$.
We also put
\begin{align*}
	\varphi_0(y) = (4\pi)^{-n/2} \exp \left( -\frac{|y|^2}{4} \right).
\end{align*}
Then, it is easy to see that
\begin{align}
\label{varphi0_int}
	\int_{\mathbb{R}^n}\varphi_0(y) dy = 1
\end{align}
and
\begin{align}
\label{phi_eq}
	\Delta_y \varphi_0 = -\frac{y}{2}\cdot \nabla_y\varphi_0-\frac{n}{2}\varphi_0.
\end{align}
We also put $\psi_0(y) = \Delta_y \varphi_0(y)$
and decompose $v, w$ as
\begin{align}
\label{sp_de_vw}
	\begin{array}{l}
	\displaystyle v(s,y) = \alpha(s) \varphi_0(y) + f(s,y),\\[5pt]
	\displaystyle w(s,y) = \frac{d\alpha}{ds}(s) \varphi_0(y) + \alpha(s)\psi_0(y) + g(s,y),
	\end{array}
\end{align}
where we expect that $(f,g)$ can be regarded as remainder terms.
In order to derive the system that $(f,g)$ satisfies,
we first note the following lemma.
\begin{Lemma}{\cite[Lemma 3.8]{Wa17JMAA}}\label{lem_alpha}
We have
\begin{align}
\label{alpha_dt}
	\frac{d\alpha}{ds}(s) &= \int_{\mathbb{R}^n}w(s,y)dy,\\
\label{alpha_ddt}
	\frac{e^{-s}}{b(t(s))^2}\frac{d^2\alpha}{ds^2}(s)
	&= \frac{e^{-s}}{b(t(s))^2} \frac{d\alpha}{ds}(s)
		- \frac{d\alpha}{ds}(s) + \int_{\mathbb{R}^n}r(s,y)dy,
\end{align}
where $r$ is defined by \eqref{r}.
\end{Lemma}
From the system \eqref{eq_vw}, Lemma \ref{lem_alpha}
and the equation \eqref{phi_eq},
we see that $f$ and $g$ satisfy the following system:
\begin{align}
\label{eq_fg}
	\left\{\begin{array}{ll}
	\displaystyle
	f_s - \frac{y}{2}\cdot \nabla_yf-\frac{n}{2}f = g,&s>0, y\in\mathbb{R}^n,\\[5pt]
	\displaystyle
	\frac{e^{-s}}{b(t(s))^2}\left( g_s - \frac{y}{2}\cdot\nabla_y g -\left(\frac{n}{2}+1\right) g \right)
		+ g = \Delta_y f + h,&s>0,y\in\mathbb{R}^n,\\[5pt]
	f(0,y) = v_0(y)-\alpha(0)\varphi_0(y),&y\in\mathbb{R}^n,\\[5pt]
	g(0,y) = w_0(y)-\frac{d\alpha}{ds}(0)\varphi_0(y)-\alpha(0)\psi_0(y),
		&y\in\mathbb{R}^n,
	\end{array}\right.
\end{align}
where $h$ is given by
\begin{align}
\nonumber
	h(s,y) &= \frac{e^{-s}}{b(t(s))^2}
		\left( -2 \frac{d\alpha}{ds}(s) \psi_0(y)
		+\alpha(s)
		\left(\frac{y}{2}\cdot\nabla_y\psi_0(y)
			+\left(\frac{n}{2}+1\right)\psi_0(y) \right) \right)\\
\label{h}
		&\quad + r(s,y) - \left(\int_{\mathbb{R}^n} r(s,y) dy \right) \varphi_0(y).
\end{align}
Moreover, from \eqref{alpha}, \eqref{varphi0_int} and \eqref{alpha_dt}, it follows that
\begin{align}
\label{fg_int}
	\int_{\mathbb{R}^n}f(s,y)dy = \int_{\mathbb{R}^n}g(s,y)dy = 0.
\end{align}
We also notice that the condition \eqref{fg_int} implies
\begin{align}
\label{h_int}
	\int_{\mathbb{R}^n} h(s,y) dy = 0.
\end{align}

\subsection{Energy estimates for $n=1$}\ 

To obtain the decay estimates for $f,g$, we introduce
\begin{align}
\label{1_FG}
	F(s,y) = \int_{-\infty}^y f(s,z)dz,\quad
	G(s,y) = \int_{-\infty}^y g(s,z)dz.
\end{align}
From the following lemma and the condition \eqref{fg_int},
we see that $F,G\in C([0,S); L^2(\mathbb{R}))$. 
\begin{Lemma}[Hardy-type inequality]{\cite[Lemma 3.9]{Wa17JMAA}}\label{lem_hardy}
Let $f=f(y)$ belong to $H^{0,1}(\mathbb{R})$ and satisfy
$\int_{\mathbb{R}}f(y)dy = 0$,
and let $F(y) = \int_{-\infty}^y f(z)dz$.
Then, we have
\begin{align}
\label{hardy}
	\int_{\mathbb{R}}F(y)^2 dy \le 4 \int_{\mathbb{R}}y^2 f(y)^2 dy.
\end{align}
\end{Lemma}
Since $f$ and $g$ satisfy the equation \eqref{eq_fg}, we can show that
$F$ and $G$ satisfy the following system:
\begin{align}
\label{eq_FG}
	\left\{\begin{array}{ll}
	\displaystyle F_s-\frac{y}{2}F_y = G,&s>0,y\in\mathbb{R},\\[5pt]
	\displaystyle
	\frac{e^{-s}}{b(t(s))^2}\left( G_s - \frac{y}{2}G_y -G \right) + G
	= F_{yy} + H,&s>0, y\in \mathbb{R},\\[5pt]
	\displaystyle F(0,y) = \int_{-\infty}^{y}f(0,z)dz,\ 
	G(0,y) = \int_{-\infty}^{y}g(0,z)dz, &y\in \mathbb{R},
	\end{array}\right.
\end{align}
where
\begin{align}
\label{H}
	H(s,y) = \int_{-\infty}^yh(s,z)dz.
\end{align}
We define the following energy.
\begin{align*}
	E_0(s) &= \int_{\mathbb{R}} \left( \frac{1}{2}\left( F_y^2 + \frac{e^{-s}}{b(t(s))^2}G^2 \right)
		+ \frac{1}{2}F^2 + \frac{e^{-s}}{b(t(s))^2} FG \right) dy,\\
	E_1(s) &= \int_{\mathbb{R}} \left( \frac{1}{2} \left( f_y^2 + \frac{e^{-s}}{b(t(s))^2}g^2 \right)
		+ f^2 + 2\frac{e^{-s}}{b(t(s))^2}fg \right)dy,\\
	E_2(s) & = \int_{\mathbb{R}} y^2 \left[ \frac{1}{2} \left( f_y^2 + \frac{e^{-s}}{b(t(s))^2}g^2 \right)
		+ \frac{1}{2} f^2 + \frac{e^{-s}}{b(t(s))^2}fg \right] dy,\\
	E_3(s) &= \frac{1}{2}\frac{e^{-s}}{b(t(s))^2}\left( \frac{d\alpha}{ds}(s) \right)^2
		+ e^{-2\lambda s}\alpha(s)^2,\\
	E_4(s) &= \frac{1}{2}\alpha(s)^2
		+ \frac{e^{-s}}{b(t(s))^2}\alpha(s) \frac{d\alpha}{ds}(s)
\end{align*}
and
\begin{align*}
	E_5(s) = \sum_{j=0}^4 C_j E_j(s),
\end{align*}
where $\lambda$ is a parameter such that
$0 < \lambda \le 1/4$
and $C_j\ (j=0,\ldots, 4)$ are constants such that
$C_2 = C_3 = C_4 =1$ and
$1 \ll C_1 \ll C_0$.
Then, we have the following energy estimates.
\begin{Lemma}{\cite[Lemmas 3.10--3.17]{Wa17JMAA}}\label{lem_en0}
We have
\begin{align*}
	\frac{d}{ds}E_j(s)
	+ \delta_j E_j(s) + L_j(s)
	= R_j(s),
\end{align*}
for $j = 0, \ldots, 4$,
where
$\delta_j= \frac12\ (j=0,1,2)$,
$\delta_3 = 2\lambda$,
$\delta_4 = 0$,
and
\begin{align*}
	L_0(s) &= \int_{\mathbb{R}}\left( \frac{1}{2}F_y^2 + G^2 \right)dy,\\
	L_1(s) &= \int_{\mathbb{R}}\left( f_y^2 + g^2 \right) dy
			- \int_{\mathbb{R}}f^2 dy,\\
	L_2(s) &= \int_{\mathbb{R}}y^2 \left( \frac{1}{2}f_y^2 + g^2 \right)dy
		+ 2\int_{\mathbb{R}}y f_y \left( f+ g \right) dy,\\
	L_3(s) &= \left( \frac{d\alpha}{ds}(s) \right)^2,\\
	L_4(s) &= 0
\end{align*}
and
\begin{align*}
	R_0(s) &= \frac{3}{2}\frac{e^{-s}}{b(t(s))^2}\int_{\mathbb{R}}G^2 dy 
		- \frac{b^{\prime}(t(s))}{b(t(s))^2}\int_{\mathbb{R}}\left( G^2 + 2FG \right)dy
		+ \int_{\mathbb{R}} (F+G)H dy,\\
	R_1(s) &= 3\frac{e^{-s}}{b(t(s))^2}\int_{\mathbb{R}}g^2 dy
		+ 2\frac{e^{-s}}{b(t(s))^2}\int_{\mathbb{R}}fg dy
		- \frac{b^{\prime}(t(s))}{b(t(s))^2}\int_{\mathbb{R}}(g^2+4fg)dy \\
		&\quad + \int_{\mathbb{R}} \left(2f+g\right) h dy,\\
	R_2(s) &= \frac{3}{2}\frac{e^{-s}}{b(t(s))^2}\int_{\mathbb{R}}y^2  g^2 dy
		-\frac{b^{\prime}(t(s))}{b(t(s))^2}\int_{\mathbb{R}}y^2 (2f+g)g dy
		+\int_{\mathbb{R}}y^2 (f+g)hdy,\\
	R_3(s) &=  \frac{1}{2}(2\lambda +1 ) \frac{e^{-s}}{b(t(s))^2}\left( \frac{d\alpha}{ds}(s) \right)^2
			- \frac{b^{\prime}(t(s))}{b(t(s))^2}\left( \frac{d\alpha}{ds}(s) \right)^2 \\
		&\quad + \frac{d\alpha}{ds}(s) \left( \int_{\mathbb{R}^n} r(s,y) dy \right)
		+ 2 e^{-2 \lambda s} \alpha(s) \frac{d\alpha}{ds}(s),\\
	R_4(s) &= \frac{e^{-s}}{b(t(s))^2}\left( \frac{d\alpha}{ds}(s) \right)^2
		- 2 \frac{b^{\prime}(t(s))}{b(t(s))^2}\alpha(s) \frac{d\alpha}{ds}(s)
		+\alpha(s) \left( \int_{\mathbb{R}^n} r(s,y)dy \right).
\end{align*}
Moreover, we have
\begin{align*}
	\frac{d}{ds}E_5(s) + 2\lambda \sum_{j=0}^3 C_j E_j(s) + L_5(s) = R_5(s),
\end{align*}
where
\begin{align*}
	L_5(s) &= \sum_{j=0}^2 \left[ \left( \frac{1}{2} - 2 \lambda \right) C_j E_j(s)
			+ C_j L_j(s) \right]
			+ C_3 L_3(s)
\end{align*}
and
\begin{align*}
	R_5(s) = \sum_{j=0}^4 C_j R_j(s).
\end{align*}
Furthermore, there exist
$C_0 > C_1 > 1$ and $s_0 > 0$ such that
\begin{align*}
	&\| f(s) \|_{H^{1,1}}^2 + \| g(s) \|_{H^{0,1}}^2 + \left( \frac{d\alpha}{ds}(s) \right)^2
	\le C L_5(s),\\
	&\| f(s) \|_{H^{1,1}}^2 + \frac{e^{-s}}{b(t(s))^2} \| g(s) \|_{H^{0,1}}^2
		+ \alpha(s)^2 +  \frac{e^{-s}}{b(t(s))^2} \left( \frac{d\alpha}{ds}(s) \right)^2
	\le C E_5(s)
\end{align*}
and
\begin{align*}
	|R_5(s)| &\le \frac12 L_5(s)
		+ C e^{-\frac{1-\beta}{1+\beta} s} E_5(s)
		+ C e^{n(p_F-p)s} E_5(s)^p
		+ C e^{\frac{n}{2}(p_F-p) s} E_5(s)^{\frac{p+1}{2}}
\end{align*}
are valid for $s \ge s_0$.
\end{Lemma}%

\subsection{Energy estimates for $n \ge 2$}\ 

When $n\ge 2$, we cannot consider primitives.
Instead of them, we define
\begin{align*}
	\hat{F}(s,\xi) = |\xi|^{-n/2-\delta}\hat{f}(s,\xi),\quad
	\hat{G}(s,\xi) = |\xi|^{-n/2-\delta}\hat{g}(s,\xi),\quad
	\hat{H}(s,\xi) = |\xi|^{-n/2-\delta}\hat{h}(s,\xi),
\end{align*}
where
$0<\delta<1$,
and
$\hat{f}(s,\xi)$ denotes the Fourier transform of $f(s,y)$ with respect to
the space variable.

By virtue of the cancelation conditions \eqref{fg_int}, \eqref{h_int},
$\hat{F}, \hat{G}, \hat{H}$ make sense as $L^2$-functions:
\begin{Lemma}{\cite[Lemma 3.11]{Wa17JMAA}}\label{lem_hardy2}
Let $m>n/2+1$ and $f(y) \in H^{0,m}(\mathbb{R}^n)$ be a function satisfying
$\hat{f}(0) = (2\pi)^{-n/2} \int_{\mathbb{R}^n}f(y)dy = 0$.
Let
$\hat{F}(\xi) = |\xi|^{-n/2-\delta}\hat{f}(\xi)$
with some $0<\delta<1$.
Then,
there exists a constant $C(n,m,\delta)>0$ such that
\begin{align}
\label{hardy2}
	\| F \|_{L^2} \le C(n,m,\delta) \| f \|_{H^{0,m}}
\end{align}
holds.
\end{Lemma}
We also notice that
$\| f \|_{L^2}$
can be controlled by the terms
$\| \nabla f \|_{L^2}$ and $\| \nabla F \|_{L^2}$,
which come from the diffusion.
\begin{Lemma}{\cite[(3.39)]{Wa17JMAA}}\label{lem_f}
In addition to the assumptions in Lemma \ref{lem_hardy2},
we further assume $f \in H^{1}(\mathbb{R}^n)$.
Then, for any small $\eta > 0$, there exists a constant $C>0$ such that
we have
\begin{align*}
	\| f \|_{L^2}^2 \le \eta \| \nabla f \|_{L^2}^2 + C \| \nabla F \|_{L^2}^2
\end{align*}
holds.
\end{Lemma}

In this case
$\hat{F}$ and $\hat{G}$ satisfy the following system.
\begin{align*}
	\left\{ \begin{array}{ll}
	\displaystyle \hat{F}_s + \frac{\xi}{2}\cdot \nabla_{\xi}\hat{F}
		+\frac{1}{2}\left( \frac{n}{2} + \delta \right) \hat{F} = \hat{G},
	&s>0, \xi \in \mathbb{R}^n,\\
	\displaystyle \frac{e^{-s}}{b(t(s))^2}\left( \hat{G}_s + \frac{\xi}{2}\cdot \nabla_{\xi} \hat{G}
		+ \frac{1}{2} \left( \frac{n}{2}+\delta-2 \right) \hat{G} \right) + \hat{G}
		= -|\xi|^2 \hat{F} + \hat{H},
	&s>0, \xi\in\mathbb{R}^n.
	\end{array} \right.
\end{align*}

We define the following energy
\begin{align*}
	E_0(s) &= {\rm Re} \int_{\mathbb{R}^n}
		\left( \frac{1}{2}\left( |\xi|^2 |\hat{F}|^2 + \frac{e^{-s}}{b(t(s))^2} |\hat{G}|^2 \right)
			+ \frac{1}{2} |\hat{F}|^2 + \frac{e^{-s}}{b(t(s))^2}\hat{F} \bar{\hat{G}}  \right) d\xi,\\
	E_1(s) &= \int_{\mathbb{R}^n}
		\left( \frac{1}{2}\left( |\nabla_y f |^2 + \frac{e^{-s}}{b(t(s))^2}g^2 \right)
		+ \left( \frac{n}{4} + 1 \right)
			\left( \frac{1}{2} f^2 + \frac{e^{-s}}{b(t(s))^2}fg \right) \right) dy,\\
	E_2(s) &= \int_{\mathbb{R}^n}
		|y|^{2m} \left[
			\frac{1}{2}\left( |\nabla_y f|^2 + \frac{e^{-s}}{b(t(s))^2}g^2 \right)
				+\frac{1}{2}f^2 + \frac{e^{-s}}{b(t(s))^2}fg \right] dy,\\
	E_3(s) &= \frac{1}{2}\frac{e^{-s}}{b(t(s))^2}\left( \frac{d\alpha}{ds}(s) \right)^2
		+ e^{-2\lambda s}\alpha(s)^2,\\
	E_4(s) &= \frac{1}{2}\alpha(s)^2
		+ \frac{e^{-s}}{b(t(s))^2}\alpha(s) \frac{d\alpha}{ds}(s)
\end{align*}
and
\begin{align*}
	E_5(s) = \sum_{j=0}^4 C_j E_j(s),
\end{align*}
where $\lambda$ is a parameter such that
$0 < \lambda < \min\{ \frac12, \frac{m}{2}-\frac{n}{4} \}$
and $C_j\ (j=0,\ldots, 4)$ are constants such that
$C_2 = C_3 = C_4 =1$ and
$1 \ll C_1 \ll C_0$.
Then, we have the following energy estimates.
\begin{Lemma}{\cite[Lemmas 3.12--3.17]{Wa17JMAA}}\label{lem_en2}
We have
\begin{align*}
	\frac{d}{ds}E_j(s)
	+\delta_j E_j(s) + L_j(s)
	= R_j(s),
\end{align*}
for $j = 0, \ldots, 4$,
where
$\delta_0 = \delta_1 = \delta$,
$\delta_2 = m -\frac{n}{2} -\eta$,
$\delta_3 = 2 \lambda$,
$\delta_4 = 0$,
and
$\eta$ is a small parameter such that
$0<\eta < m-\frac{n}{2}$,
and
\begin{align*}
	L_0(s) &= \frac{1}{2} \int_{\mathbb{R}^n} |\xi|^2 |\hat{F}|^2 d\xi
		+\int_{\mathbb{R}^n} |\hat{G}|^2 d\xi,\\
	L_1(s) &= \frac{1}{2}(1-\delta) \int_{\mathbb{R}^n}|\nabla_y f|^2 dy
		+ \int_{\mathbb{R}^n} g^2 dy
		- \left( \frac{n}{4}+\frac{\delta}{2} \right)
			\left(\frac{n}{4}+1\right) \int_{\mathbb{R}^n}f^2 dy,\\
	L_2(s) &= \frac{\eta}{2}\int_{\mathbb{R}^n}|y|^{2m} f^2 dy
	 + \frac{1}{2} ( \eta + 1) \int_{\mathbb{R}^n} |y|^{2m} |\nabla_y f|^2 dy
	 + \int_{\mathbb{R}^n} |y|^{2m} g^2 dy\\
	 &\quad + 2m \int_{\mathbb{R}^n}|y|^{2m-2} (y\cdot \nabla_y f)(f+g) dy,\\
	L_3(s) &= \left( \frac{d\alpha}{ds}(s) \right)^2,\\
	L_4(s) &= 0
\end{align*}
and
\begin{align*}
	R_0(s) &= \frac{3}{2} \frac{e^{-s}}{b(t(s))^2}\int_{\mathbb{R}^n} |\hat{G}|^2d\xi
	- \frac{b^{\prime}(t(s))}{b(t(s))^2}
		{\rm Re}\int_{\mathbb{R}^n} \left( 2\hat{F} +\hat{G} \right) \bar{\hat{G}} d\xi \\
		&\quad + {\rm Re} \int_{\mathbb{R}^n} \left( \hat{F} + \hat{G} \right) \bar{\hat{H}} d\xi,\\
	R_1(s) &= \left( \frac{n}{2}+\delta \right)\left( \frac{n}{4}+1 \right)
			\frac{e^{-s}}{b(t(s))^2}\int_{\mathbb{R}^n} fg dy
		+ \frac{1}{2}(n+3+\delta) \frac{e^{-s}}{b(t(s))^2} \int_{\mathbb{R}^n} g^2 dy \\
	&\quad
		- \frac{b^{\prime}(t(s))}{b(t(s))^2}
			\int_{\mathbb{R}^n} \left( 2 \left(\frac{n}{4}+1\right) f + g \right) g dy
		+ \int_{\mathbb{R}^n} \left( \left( \frac{n}{4} + 1 \right) f + g \right) h dy,\\
	R_2(s) &= -\eta \frac{e^{-s}}{b(t(s))^2} \int_{\mathbb{R}^n} |y|^{2m} f g dy
		- \frac{1}{2}(\eta - 3) \frac{e^{-s}}{b(t(s))^2} \int_{\mathbb{R}^n} |y|^{2m} g^2 dy\\
	&\quad - \frac{b^{\prime}(t(s))}{b(t(s))^2} \int_{\mathbb{R}^n}|y|^{2m} (2f +g) g dy
		+ \int_{\mathbb{R}^n} |y|^{2m} (f+g) h dy,\\
	R_3(s) &=  \frac{1}{2}(2\lambda +1 ) \frac{e^{-s}}{b(t(s))^2}\left( \frac{d\alpha}{ds}(s) \right)^2
			- \frac{b^{\prime}(t(s))}{b(t(s))^2}\left( \frac{d\alpha}{ds}(s) \right)^2 \\
		&\quad + \frac{d\alpha}{ds}(s) \left( \int_{\mathbb{R}^n} r(s,y) dy \right)
		+ 2 e^{-2 \lambda s} \alpha(s) \frac{d\alpha}{ds}(s),\\
	R_4(s) &= \frac{e^{-s}}{b(t(s))^2}\left( \frac{d\alpha}{ds}(s) \right)^2
		- 2 \frac{b^{\prime}(t(s))}{b(t(s))^2}\alpha(s) \frac{d\alpha}{ds}(s)
		+\alpha(s) \left( \int_{\mathbb{R}^n} r(s,y)dy \right).
\end{align*}
Moreover, we have
\begin{align*}
	\frac{d}{ds}E_5(s) + 2\lambda \sum_{j=0}^3 C_j E_j(s) + L_5(s) = R_5(s),
\end{align*}
where
\begin{align*}
	L_5(s) &= C_0 (\delta-2\lambda) E_0(s)
		+ C_1 (\delta-2\lambda) E_1(s) + (m-\frac{n}{2}-\eta-2\lambda) E_2(s) \\
		&\quad + \sum_{j=0}^4 C_j L_j(s)
\end{align*}
and
\begin{align*}
	R_5(s) = \sum_{j=0}^4 C_j R_j(s).
\end{align*}
Furthermore, there exist
$C_0 > C_1 > 1$ and $s_0 > 0$ such that
\begin{align*}
	&\| f(s) \|_{H^{1,m}}^2 + \| g(s) \|_{H^{0,m}}^2 + \left( \frac{d\alpha}{ds}(s) \right)^2
	\le C L_5(s),\\
	&\| f(s) \|_{H^{1,m}}^2 + \frac{e^{-s}}{b(t)^2} \| g(s) \|_{H^{0,m}}^2
		+ \alpha(s)^2 +  \frac{e^{-s}}{b(t)^2} \left( \frac{d\alpha}{ds}(s) \right)^2
	\le C E_5(s)
\end{align*}
and
\begin{align*}
	|R_5(s)| &\le \frac12 L_5(s)
		+ C e^{-\frac{1-\beta}{1+\beta} s} E_5(s)
		+ C e^{n(p_F-p)s} E_5(s)^p
		+ C e^{\frac{n}{2}(p_F-p) s} E_5(s)^{\frac{p+1}{2}}
\end{align*}
are valid for $s \ge s_0$.
\end{Lemma}%

\subsection{A priori estimate and the proof of Proposition \ref{prop_low}}\ 

By Lemmas \ref{lem_en0} and \ref{lem_en2}
with taking $0< \lambda < \min\{ \frac12, \frac{\delta}{2}, \frac{m}{2}-\frac{n}{4}\}$
and $\eta$ sufficiently small if $n \ge 2$,
we can see that
$(f,g)$ satisfies the following a priori estimate for
$s \ge s_0$.
Here we note that
the local solution exists for $s > s_0$,
provided that $\varepsilon$ is sufficiently small
by Proposition \ref{prop_loc2}.
\begin{Lemma}{\cite[(3.53)]{Wa17JMAA}}\label{lem_e5}
There exists $s_0 > 0$ such that for $s \ge s_0$, we have
\begin{align}
\label{e5est}
	\frac{d}{ds}E_5(s)
	\le C e^{-\frac{1-\beta}{1+\beta} s} E_5(s)
		+ C e^{n(p_F-p)s} E_5(s)^p
		+ C e^{\frac{n}{2}(p_F-p) s} E_5(s)^{\frac{p+1}{2}}
\end{align}
(where we interpret $1/(1+\beta)$ as an arbitrarily large number
when $\beta = -1$).
\end{Lemma}

Now we are in a position to prove Proposition \ref{prop_low}.
\begin{proof}[Proof of Proposition \ref{prop_low}]
Let $\varepsilon_1 > 0$ be sufficiently small so that
the local solution $(v,w)$ of \eqref{eq_vw} exists for 
$s > s_0$
(see Proposition \ref{prop_loc2}).
Therefore, by Lemma \ref{lem_e5}, we see that
$(f,g)$ satisfies the a priori estimate \eqref{e5est}.
We put
\[
	\Lambda (s) := \exp \left( -C\int_{s_0}^s e^{- \frac{1-\beta}{1+\beta} \tau}\, d\tau \right)
\]
(where we interpret $1/(1+\beta)$ as an arbitrarily large number
when $\beta = -1$).
We note that
$c_0 \le \Lambda(s) \le 1$ holds for some $c_0 > 0$,
and $\Lambda(s_0) = 1$.
Multiplying \eqref{e5est} by $\Lambda(s)$
and integrating it over $[s_0, s]$, we see that
\[
	\Lambda(s) E_5(s) \le E_5(s_0)
		+ C \int_{s_0}^s
		\left[ \Lambda(\tau) e^{n(p_F-p) \tau} E_5(\tau)^p
			+ \Lambda(\tau) e^{\frac{n}{2}(p_F-p) \tau} E_5(\tau)^{\frac{p+1}{2}}
						\right] \, d\tau
\]
holds for $s_0 \le s < S_0(\varepsilon)$.
Putting
\[
	M(s) := \sup_{s_0 \le \tau \le s} E_5(\tau)
\]
and noting
\[
	M(s_0) \le C(s_0) \varepsilon^2 \| (a_0, a_1) \|_{H^{1,m}\times H^{0,m}}^2,
\]
which can be easily proved by local existence result
(see the proof of \cite[Proposition 3.5]{Wa17JMAA}),
we have
\begin{align}
\label{estm}
	M(s) \le
		C_0^{\prime} \varepsilon^2 I_0
		+ C_0^{\prime}
			\left( e^{n(p_F-p)s} M(s)^p + e^{\frac{n}{2}(p_F-p)s } M(s)^{\frac{p+1}{2}} \right)
\end{align}
for $s_0 \le s < S_0(\varepsilon)$ and some $C_0^{\prime} > 0$,
where
$I_0 = \| (a_0, a_1) \|_{H^{1,m}\times H^{0,m}}^2$.
Let
$S_1 = S_1(\varepsilon) \ge s_0$ is the first time
such that $M$ attains the value
\[
	M(S_1) = 2 C_0^{\prime} \varepsilon^2 I_0.
\]
We note that if $S_0(\varepsilon) = \infty$,
then Proposition \ref{prop_low} obviously holds, 
and if $S_0(\varepsilon) < \infty$,
then such $S_1$ actually exists because
$\lim_{s \to S_0(\varepsilon)} M(s) = \infty$.
Thus, in what follows we assume $S_0(\varepsilon) < \infty$.
Then, substituting $s= S_1$ in \eqref{estm}, we see that
\begin{align*}
	C_0^{\prime} \varepsilon^2 I_0
	&\le 2C_0^{\prime}
	\max\left\{ e^{n(p_F-p)S_1} (2C_0^{\prime} \varepsilon^2 I_0)^p,
		e^{\frac{n}{2}(p_F-p)S_1 } (2C_0^{\prime} \varepsilon^2 I_0)^{\frac{p+1}{2}}
		\right\}.
\end{align*}
No matter which quantity attains the maximum,
we obtain
\[
	\varepsilon^{-\frac{2(p-1)}{n(p_F-p)}}  \le C e^{S_1}.
\]
Thus, we conclude
\[
	\varepsilon^{-\frac{1}{\frac{1}{p-1} - \frac{n}{2}}}
	\le C \left( B(T_0(\varepsilon)) + 1 \right).
\]
This and the definition of $B(t)$ lead to the
desired estimate, and we finish the proof.
\end{proof}

\begin{proof}[Proof of Proposition \ref{prop_lowcr}]
In the same way to the derivation of \eqref{estm},
noting $p=p_F$, we have
\begin{align}
\label{estmcr}
	M(s) \le
		C_0^{\prime} \varepsilon^2 I_0
		+ C_0^{\prime}
			(s-s_0) \left( M(s)^p + M(s)^{\frac{p+1}{2}} \right)
\end{align}
for $s_0 \le s < S_0(\varepsilon)$ and some $C_0^{\prime} > 0$,
Let
$S_1 = S_1(\varepsilon) \ge s_0$ is the first time
such that $M$ attains the value
\[
	M(S_1) = 2 C_0^{\prime} \varepsilon^2 I_0.
\]
Moreover,
we take
$\varepsilon_2 \le \varepsilon_1$ further small so that
$2 C_0^{\prime} \varepsilon^2 I_0 \le 1$
holds for $\varepsilon \in (0,\varepsilon_2]$.
Then, it is obvious that
$M(S_1)^p \le M(S_1)^{\frac{p+1}{2}}$ for $\varepsilon \in (0,\varepsilon_2]$
and hence, we eventually obtain
\[
	2 C_0^{\prime} \varepsilon^2 I_0
	\le C_0^{\prime} \varepsilon^2 I_0
		+ 2C_0^{\prime} (S_1 - s_0) \left( 2 C_0^{\prime} \varepsilon^2 I_0 \right)^{\frac{p+1}{2}}.
\]
This implies
\[
	\exp \left( C \varepsilon^{-(p-1)} + s_0 \right) \le B(T_0) + 1.
\]
Therefore, by the definition of $B(t)$,
we have the desired estimate.
\end{proof}

\section*{Appendix}
Here, we prove existence of a unique strong solution
in the sense of Definition \ref{def_sol}
for the Cauchy problem \eqref{eq:1}.
\begin{Proposition}\label{prop_le}
Let
$p$
satisfy
$1<p <\infty \ (n=1,2)$, $1<p \le n/(n-2) \ (n\ge 3)$.
We assume that
$b(t)$ is a smooth nonnegative function.
Let
$(u_0, u_1) \in H^1(\mathbb{R}^n) \times L^2(\mathbb{R}^n)$.
Then, there exist a constant $T>0$ and a unique strong solution $u$
of the Cauchy problem \eqref{eq:1} on $[0,T)$.
Moreover, we have the blow-up alternative, that is,
for the lifespan $T_0$ defined in Definition \ref{def_ls},
if $T_0 < \infty$, then
\begin{align*}%
	\lim_{t \to T_0-} \| (u, u_t)(t) \|_{H^1\times L^2} = \infty.
\end{align*}%
\end{Proposition}
\begin{proof}
Let $(u_0, u_1) \in H^1(\mathbb{R}^n) \times L^2(\mathbb{R}^n)$
and let
$F \in L^1_{{\rm loc}} ([0,\infty); L^2(\mathbb{R}^n))$.
First, we recall that the strong solution of the linear problem
\begin{align}%
\label{lin_dw}
	\left\{ \begin{array}{ll}
		\square u + b(t) u_t = F,&t>0, x \in \mathbb{R}^n,\\
		u(0) = u_0, \ u_t(0) = u_1,& x\in \mathbb{R}^n
	\end{array}\right.
\end{align}%
can be easily obtained via the Fourier transform and the Duhamel principle,
and the corresponding solution belongs to
$C^1([0,T) ; L^2(\mathbb{R}^n)) \cap C ([0,T) ; H^1(\mathbb{R}^n))$.
Moreover, by an approximation argument, we see that the solution $u$
satisfies the energy identity
\begin{align*}%
	&\frac12 \| (\nabla_xu, u_t)(t) \|_{L^2}^2 
	+ \int_0^t \int_{\mathbb{R}^n} b(\tau) u_t(\tau, x)^2\,dxd\tau \\
	&\quad = \frac12 \| (\nabla_x u_0, u_1) \|_{L^2}^2
	+ \int_0^t \int_{\mathbb{R}^n} F(\tau, x) u_t(\tau,x) \,dxd\tau
\end{align*}%
for $t > 0$.
Combining this with a Gronwall-type inequality (see \cite[Lemma 9.12]{Wa_thesis})
and the assumption $b(t) \ge 0$,
we further obtain
\begin{align}%
\label{en_es}
	\| (\nabla_xu, u_t)(t) \|_{L^2}
	\le \| (\nabla_x u_0, u_1) \|_{L^2}
		+ \int_0^t \| F(\tau) \|_{L^2} \,dxd\tau
\end{align}%
for $t > 0$.
To construct a solution of the Cauchy problem \eqref{eq:1},
we employ the contraction mapping principle.
To this end, we take a constant $R>0$ such that
$ \| (u_0, u_1) \|_{H^1 \times L^2} \le R$
and define
\begin{align*}%
	&K(T,R) \\
	&:=
	\left\{ v \in C^1([0,T) ; L^2(\mathbb{R}^n)) \cap C ([0,T) ; H^1(\mathbb{R}^n))
	\,;\, \sup_{t\in [0,T)}\| (v, v_t )(t) \|_{H^1 \times L^2} \le 3R \right\}.
\end{align*}%
We define the metric
\begin{align*}%
	d(u,v) := \sup_{t\in [0,T)} \| (u-v, u_t - v_t)(t) \|_{H^1\times L^2}
\end{align*}%
for $u, v \in K(T,R)$.
Then, $K(T,R)$ is a complete metric space with the metric $d(\cdot, \cdot)$.
Let $1<p <\infty\ (n=1,2)$, $1<p \le n/(n-2)\ n\ge 3$.
Let
$u^{(0)}$
be the solution of the linear problem \eqref{lin_dw} with $F=0$.
Then, for $j =1, 2, \ldots$, we successively define
$u^{(j)}$
as the solution of the linear problem \eqref{lin_dw} with $F = |u^{(j-1)}|^p$.
By \eqref{en_es}, the first approximation
$u^{(0)}$ clearly belongs to $K(T,R)$.
Furthermore, by the Sobolev embedding theorem,
we see that if $u^{(j-1)} \in K(T,R)$, then
\begin{align*}%
	\| (\nabla u^{(j)}, u^{(j)}_t)(t) \|_{L^2}
	&\le R + \int_0^t \| u^{(j-1)}(\tau) \|_{L^{2p}}^p\,d\tau \\
	&\le R + C \int_0^t \| u^{(j-1)}(\tau) \|_{H^1}^p\, d\tau \\
	&\le R + C T R^p.
\end{align*}%
We also estimate
\begin{align*}%
	\| u^{(j)} (t) \|_{L^2} &\le \| u_0 \|_{L^2} + \int_0^t \| u^{(j)}_t (\tau) \|_{L^2} \,d\tau \\
	&\le R + T(R + C T R^p).
\end{align*}%
Therefore, taking $T>0$ sufficiently small so that
$C T R^p + T(R + C T R^p) \le R$ holds, we obtain
$u^{(j)} \in K(T,R)$.
Thus, it follows from the mathematical induction that
$u^{(j)} \in K(T,R)$ for all $j \ge 0$.
Next, making use of
\begin{align*}%
	| |u|^p - |v|^p | \le C (|u| + |v|)^{p-1} |u-v|,
\end{align*}%
we estimate
\begin{align*}%
	&\| (\nabla u^{(j)}(t) - \nabla u^{(j-1)}(t), u^{(j)}_t(t) - u^{(j-1)}_t(t)) \|_{L^2} \\
	&\quad \le C \int_0^t
	\| (|u^{(j-1)}(\tau)| + |u^{(j-2)}(\tau)| )^{p-1}
	|u^{(j-1)}(\tau) - u^{(j-2)}(\tau)| \|_{L^2}\, d\tau\\
	&\quad \le C \int_0^t \| (|u^{(j-1)}(\tau)| + |u^{(j-2)}(\tau)| ) \|_{L^{2p}}^{p-1}
					\| u^{(j-1)}(\tau) - u^{(j-2)}(\tau) \|_{L^{2p}}\, d\tau \\
	&\quad \le C \int_0^t
			( \| u^{(j-1)}(\tau) \|_{H^1} + \| u^{(j-2)}(\tau) \|_{H^1} )^{p-1}
			\| u^{(j-1)}(\tau) - u^{(j-2)}(\tau) \|_{H^1}\,d\tau \\
	&\quad \le C T R^{p-1} d(u^{(j-1)}, u^{(j-2)}).
\end{align*}%
We also have
\begin{align*}%
	\| u^{(j)}(t) - u^{(j-1)}(t) \|_{L^2}
	&\le \int_0^t \| u^{(j)}_t (\tau) - u^{(j-1)}_t(\tau) \|_{L^2}\,d\tau \\
	&\le CT^2 R^{p-1} d(u^{(j-1)}, u^{(j-2)}).
\end{align*}%
Therefore, taking $T >0$ further small so that
$CTR^{p-1} + CT^2 R^{p-1} \le \kappa$
with some $\kappa \in (0,1)$,
we conclude
\begin{align*}%
	d(u^{(j)}, u^{(j-1)}) \le \kappa d(u^{(j-1)}, u^{(j-2)}).
\end{align*}%
Hence, $\{ u^{(j)} \}_{j \ge 0}$ is a Cauchy sequence in $K(T,R)$
and we find the limit $u \in K(T,R)$.
Since each $u^{(j)}$ satisfies the initial conditions
$u^{(j)}(0) = u_0$ and $u^{(j)}_t (0) = u_1$, so does $u$.
Moreover, taking the limit $j \to \infty$ in the equation
\begin{align*}%
	u^{(j)}_{tt} =  \Delta u^{(j)} - b(t) u^{(j)}_t + |u^{(j)}|^p,
\end{align*}%
which is valid in $C([0,T) ; H^{-1}(\mathbb{R}^n))$,
we see that
$u \in C^2([0,T) ; H^{-1}(\mathbb{R}^n))$
holds and
$u$ is a strong solution of \eqref{eq:1}.
Indeed, let $t_0 \in [0,T)$ and take a small neighborhood $\omega$ of $t_0$ in $[0,T)$.
Then, we have
\begin{align*}%
	\Delta u^{(j)} \to \Delta u,\quad
	b(t) u^{(j)}_t \to b(t) u_t,\quad
	|u^{(j)}|^p \to |u|^p
\end{align*}%
in $C(\omega ; H^{-1}(\mathbb{R}^n))$ as $j \to \infty$,
and this convergence is uniform in $\omega$.
Thus, there exists $w \in C(\omega ; H^{-1}(\mathbb{R}^n))$ such that
$u^{(j)}_{tt} \to w$ in $C(\omega ; H^{-1}(\mathbb{R}^n))$ as $j \to \infty$,
and this convergence is also uniform in $\omega$.
Thus, we compute
\begin{align*}%
	&\lim_{h\to 0} \left\| \frac{1}{h}( u_t(t_0 + h) - u_t(t_0)) - w(t_0) \right\|_{H^{-1}} \\
	&\quad = \lim_{h\to 0} \lim_{j \to \infty}
		\left\| \frac{1}{h}( u^{(j)}_t(t_0 + h) - u^{(j)}_t(t_0)) - u^{(j)}_{tt} (t_0) \right\|_{H^{-1}} \\
	&\quad = \lim_{j \to \infty} \lim_{h\to 0}
		\left\| \frac{1}{h}( u^{(j)}_t(t_0 + h) - u^{(j)}_t(t_0)) - u^{(j)}_{tt} (t_0) \right\|_{H^{-1}} \\
	&\quad = 0,
\end{align*}%
which leads to $u_{tt} = w$ and hence,
$u \in C^2([0,T) ; H^{-1}(\mathbb{R}^n))$.

In order to show the uniqueness, let
$u, v$ are strong solutions of \eqref{eq:1} on $[0,T)$ with the initial data $(u_0, u_1)$.
Let
$M = \sup_{t\in [0,T)}
( \| (u(t), u_t(t)) \|_{H^1\times L^2} + \| (v(t), v_t(t)) \|_{H^1\times L^2} )$.
Then, a similar argument above implies
\begin{align*}%
	&\| (u(t) - v(t), u_t(t) - v_t(t)) \|_{H^1\times L^2} \\
	&\quad \le (1+C M^{p-1})
		\int_0^t \| (u(\tau ) - v(\tau), u_t(\tau) - v_t(\tau)) \|_{H^1\times L^2}.
\end{align*}%
This and the Gronwall inequality give $u = v$.

Finally, we prove the blow-up alternative.
Let the lifespan $T_0$ be finite and we suppose that
\begin{align*}%
	\lim_{t \to T_0-} \| (u, u_t)(t) \|_{H^1\times L^2} < \infty.
\end{align*}%
Then, there exists a constant $R'>0$ such that
$\| (u, u_t)(t) \|_{H^1\times L^2} \le R'$
for any $t \in [0,T_0)$.
Let $t_0 \in [0,T_0)$ be arbitrarily fixed.
In the same way as before, we see that
there exists $T'>0$ depending only on $R'$ such that
we can construct a unique strong solution on $[t_0, t_0+T')$.
However, if $t_0 \in (T_0 - T', T_0)$, this contradicts
the definition of the lifespan.
\end{proof}

\section*{Acknowledgments}
The authors are deeply grateful to Professor Mitsuru Sugimoto
for his helpful comments.
The first, second and third authors were
partly supported by the Japan Society for the Promotion of Science,
Grant-in-Aid for JSPS Fellows No.
16J30008, 
14J01884
and
15J01600,
respectively.

\end{document}